\documentclass[12pt,leqno]{article}

\usepackage[a4paper,margin=20mm]{geometry}

\usepackage{xcolor}

\usepackage{amsmath,amsthm,amsfonts,amssymb,enumerate}
\usepackage{mathrsfs,dsfont}

\theoremstyle{plain}
\newtheorem{theorem}{Theorem}[section]
\newtheorem{corollary}[theorem]{Corollary}
\newtheorem{lemma}[theorem]{Lemma}

\theoremstyle{definition}
\newtheorem{example}[theorem]{Example}
\newtheorem*{notation}{Notation}
\newtheorem*{ack}{Acknowledgement}

\numberwithin{equation}{section}

\newcommand{\real}{{\mathds R}}
\newcommand{\nat}{{\mathds N}}
\newcommand{\Ee}{\mathds{E}}
\newcommand{\Pp}{\mathds{P}}
\newcommand{\I}{\mathds{1}}

\newcommand{\Wa}{\mathds{W}}

\newcommand{\dup}{\mathrm{d}}
\newcommand{\eup}{\mathrm{e}}

\newcommand{\Ical}{\mathcal{I}}
\newcommand{\Jcal}{\mathcal{J}}
\newcommand{\Kcal}{\mathcal{K}}
\newcommand{\Lcal}{\mathcal{L}}
\newcommand{\Lscr}{\mathscr{L}}
\newcommand{\Normal}{\mathcal{N}}
\newcommand{\Pcal}{\mathcal{P}}
\newcommand{\Pscr}{\mathscr{P}}
\newcommand{\Tcal}{\mathcal{T}}

\DeclareMathOperator{\Lip}{Lip}
\DeclareMathOperator{\Hess}{Hess}
\DeclareMathOperator{\supp}{supp}
\DeclareMathOperator{\vol}{vol}
\DeclareMathOperator{\Var}{Var}
\newcommand{\entier}[1]{\lfloor #1\rfloor}

\newcommand{\les}{\lesssim}

\begin{document}
\allowdisplaybreaks

\title{\bfseries Convergence rates for
the $p$-Wasserstein distance of the empirical measures of an ergodic Markov process
}

\author{
	Ren\'{e} L. Schilling
	\and
	Jian Wang
	\and
	Bingyao Wu
	\and
	Jie-Xiang Zhu
}

\maketitle

\begin{abstract}
Let $X:=(X_t)_{t\geq 0}$ be an ergodic Markov process on $\real^d$, and $p>0$. We derive upper bounds of the $p$-Wasserstein distance
between the invariant measure and the empirical measures of the Markov process $X$. For this we assume, e.g.\ that the transition semigroup of $X$ is
exponentially contractive in terms of the $1$-Wasserstein distance, or that the iterated Poincar\'e inequality holds together with certain moment conditions on the invariant measure. Typical examples include diffusions and underdamped Langevin dynamics.


\end{abstract}

\section{Introduction and main results}\label{sec-intro}

It is an interesting and fundamental problem in probability theory, (random) dynamical systems and numerical analysis to study the quantitative behaviour of the convergence of empirical measures in
the Wasserstein distance.

To make this precise, let us state this problem in an Euclidean setting. By $\Pscr(\real^d)$ we denote the set of all probability measures on $\real^d$, and $\mathscr{C}(\nu_0,\nu_1)\subset\Pscr(\real^d\times\real^d)$ denotes the set of couplings with marginals $\nu_0,\nu_1\in\Pscr(\real^d)$.
For $\nu_0,\nu_1 \in\Pscr(\real^d)$ and $c(x, y) := |x-y|^p$ with $x,y\in\real^d$ and $p\in (0,\infty)$, we define the \emph{optimal transport cost} as
\begin{gather*}
	\Tcal_p(\nu_0,\nu_1)
	:= \inf_{\pi\in\mathscr{C}(\nu_0,\nu_1)} \int_{\real^d\times\real^d}|x-y|^p \, \pi(\dup x,\dup y).
\end{gather*}
The ($p$-)\emph{Wasserstein \textup{(}Kantorovich\textup{)} distance} is defined as $\Wa_p(\nu_0, \nu_1) := \left[\Tcal_p(\nu_0,\nu_1)\right]^{1\wedge\frac 1p}$, see Villani \cite[Theorem 7.3]{V-2003}.

Let $(X_t)_{t\ge 0}$ be an ergodic temporally homogeneous Markov process on $\real^d$ with invariant probability measure $\mu$, and let $(P_t)_{t\ge 0}$ be its transition semigroup, i.e.\
\begin{gather*}
	P_t f (x)
	:= \Ee^x \left[ f(X_t) \right], \quad t \ge 0, \; x \in \real^d, \; f \in \mathscr{B}_b(\real^d),
\end{gather*}
where $\Ee^x$ denotes the expectation if $X_0 = x$. We write $L^p(\mu) := L^p(\real^d, \mu)$ for $1 \le p \le \infty$ with norm $\| \cdot \|_{L^p(\mu)}$. Let $\Lcal$ be the infinitesimal generator of $(P_t)_{t \ge 0}$ on $L^2(\mu)$ with domain $\mathscr D(\Lcal)$. Consider the empirical measure
\begin{gather*}
	\mu_T :=  \frac1T \int_0^T \delta_{X_t}\,\dup t, \quad T > 0,
\end{gather*}
where $\delta_{X_t}$ is the Dirac measure at $X_t$. One main problem is to establish convergence rates for $\Tcal_p(\mu_T, \mu)$ both in expectation and almost surely as $T \to \infty$.

Recently, there has been considerable progress in this problem. Following the ideas of \cite{AMB,L17}, several works \cite{TWZ, W22, WWZ} obtain bounds for $\Ee\left[\Tcal_p(\mu_T, \mu)\right]$ by regularizing the empirical measure $\mu_T$ by a symmetric Markov semigroup (which is related to the original semigroup) and comparing the Wasserstein distance with negative Sobolev norms.  This so-called PDE approach is particularly effective if $p = 2$. On the downside, this approach requires that the underlying distance is induced by a symmetric semigroup associated with the invariant measure, as well as a spectral gap -- and both conditions fail for typical degenerate models, e.g.\ the underdamped Langevin dynamics. Recall that this is the process $(X_t)_{t\ge0}: = (Y_t, Z_t)_{t\ge0} \in \real^n \times \real^n$, which is given by the following  stochastic differential equation (SDE):
\begin{gather} \label{Langevin0}
	\begin{cases}
		\dup Y_t = Z_t \, \dup t,\\
		\dup Z_t =- \big( Z_t + \nabla V(Y_t)  \big) \, \dup t + \sqrt{2 } \, \dup W_t,\\
	\end{cases}
\end{gather}
where $V \in C^2(\real^n)$ is a confining potential, and $(W_t)_{t\ge0}$ is a standard $n$-dimensional Brownian motion. The generator of this process $(X_t)_{t\ge0}$ is not uniformly elliptic. The  recent  work \cite{W25} further develops the PDE approach and obtains quantitative bounds on $\Ee \left[ \Tcal_2(\mu_T, \mu) \right]$.

In this paper we use a method that differs from the approach in \cite{W25}. Our technique is related to the strategy in \cite{FG}, which was useful for related questions in an i.i.d.\ setting. For the underdamped Langevin dynamics, an important observation is that, under suitable assumptions on
the potential
$V$, the process $(X_t)_{t\ge0}$ defined by \eqref{Langevin0} satisfies the iterated Poincar\'e inequality, which enables us to use a Bernstein-type inequality. These ingredients yield quantitative bounds for $\Tcal_p(\mu_T,\mu)$, both in expectation and almost surely, for the underdamped Langevin dynamics. Our results do not only improve the bound obtained in \cite{W25} for $p=2$, but also extend it to all $p>0$; moreover, our approach gives new almost sure bounds; see Section~\ref{Lan}.

In fact, our arguments are valid for
 general
 Markov processes that satisfy one of the following conditions: (i) the exponential contractivity condition in the $1$-Wasserstein distance, (ii) the iterated Poincar\'e inequality, (iii) the $L^2(\mu)$-coercivity. Within this setting, we improve and extend several results from \cite{CP,W22}. Besides the key lemma from \cite{FG}, our proof relies on a smoothing procedure via compactly supported smooth densities (cf.\ \cite{CP, DJL}) and on Bernstein-type inequalities (cf.\ \cite{TWZ, WWZ} for the compact setting). A new feature of our method is that it avoids any direct use of estimates involving the transition density $p_t(x, y)$ of the semigroup $(P_t)_{t\ge0}$.

\begin{notation}
Most of our notation is standard or should be clear from the context. We use $\mathscr{L}_{X}$ for the law of the random variable $X$, and we write $\Pp^\nu$ and $\Ee^\nu$ for the probability and expectation corresponding to the initial law $\nu$ respectively.  We also write $X\sim\mu$ if $\Lscr_X = \mu$.  By $\mu(f)$ we denote the integral $\int f\,\dup\mu$; in particular, $\mu(|\cdot|^q):=\int_{\real^d} |x|^q \mu(\dup x)$ is the $q$-th moment of $\mu$. We write $\I_A$ for the indicator function of a set or event $A$.

The shorthand $A \les B$ means that there is a constant $C> 0$, depending only on the parameters in the assumptions, such that $A \le C\cdot B$; if both $A \les B$ and $B \les A$ hold, we write $A \simeq B$;  $\mathds{I}_{d\times d}$ denotes the $d$-dimensional identity matrix. For any $a>0$, $\entier{a}$ denotes its integer part;
$\langle\cdot,\cdot\rangle_{L^2(\mu)}$ denotes the inner product in $L^2(\mu)$.
 Finally,  $a\wedge b$ stands for the  minimum of $a,b\in\real$, and $a_+ := \max\{a, 0\}$.
\end{notation}

Now, we introduce some assumptions and our main results.

\begin{enumerate}
\item[\textbf{(H1)}] (Exponential contractivity in the $1$-Wasserstein distance)
	There exist constants $C \ge 1$ and $\lambda_{\mathrm{E}}>0$ such that for any $\nu\in\Pscr(\real^d)$  and $t\ge0$,
	\begin{gather*}
		\Wa_1(\nu P_t, \mu P_t) \le   C \eup^{-\lambda_{\mathrm{E}} t} \, \Wa_1(\nu,\mu).
	\end{gather*}

\item[\textbf{(H2)}] (Iterated Poincar\'e inequality)
	There exists a constant $\lambda_{\mathrm{I}}> 0$ such that for any $f \in \mathscr D(\Lcal)$ with $\mu(f)  = 0$,
	\begin{gather*}
		\| \Lcal f  \|_{L^2(\mu)} \ge \lambda_{\mathrm{I}} \, \| f \|_{L^2(\mu)}.
	\end{gather*}

\item[\textbf{(H3)}] ($L^2(\mu)$-coercivity)
	There exists a constant $\lambda_{\mathrm{C}} > 0$ such that for any $f \in \mathscr D(\Lcal)$ with $\mu(f)  = 0$,
	\begin{gather*}
		\langle -\Lcal f , f     \rangle_{L^2(\mu)} \ge \lambda_{\mathrm{C}} \, \| f \|_{L^2(\mu)}^2.
	\end{gather*}
\end{enumerate}
If $\Lcal$ is symmetric, then \textbf{(H3)} is also known as the \emph{Poincar\'e inequality} or the \emph{spectral gap inequality}. In Section~\ref{subsec-more} below, we will discuss the relationships among the assumptions \textbf{(H1)}--\textbf{(H3)}, and present some typical examples.

We will now state our main results on $\Ee^\mu [\Tcal_p(\mu_T, \mu)]$.
\begin{theorem}\label{D}
	Let $(X_t)_{t\geq 0}$ be an ergodic Markov process on $\real^d$ with invariant measure $\mu$.  Assume that \textbf{\upshape(H1)} holds, let $p>0$, and assume that  $\mu$ has for some $q> \max\{p, 1\}$ a $q$-th moment
		$\mu(|\cdot|^q)
<\infty.$
	Define
	\begin{gather*}
		\zeta := \zeta (p, q, d) := \max \left\{ \frac{q}{q - p}, \frac{d}{p} \right\} \in (1, \infty).
	\end{gather*}
	Then  there exists a constant $C>0$ such that  for all $T \ge 2$,
	\begin{gather}\label{D-W}\begin{aligned}
		\Ee^\mu\left[\Tcal_p (\mu_T,\mu)\right]
		\leq
		 C\,
		T^{-\frac{p}{2\zeta p + 1}}&\left( (\log T)^{\frac{2d}{2d+1}} \I_{\left\{q = \frac{dp}{d - p} \right\}} + \I_{\left\{q \ne \frac{dp}{d - p}  \right\}} \right)  \left( (\log T)^{\frac{2q}{2q + (q/d) - 1}} \I_{\{ p = d  \}} + \I_{\{ p \ne d  \}}   \right).
	\end{aligned}\end{gather}
\end{theorem}

\begin{theorem}\label{IP}
	Let $(X_t)_{t\geq 0}$ be an ergodic Markov process on $\real^d$ with invariant measure $\mu$.  Assume that \textbf{\upshape(H2)} holds, let $p>0$, and assume that $\mu$ has for some $q> p$ a $q$-th moment $\mu(|\cdot|^q)  < \infty$. Define
	\begin{gather*}
		\gamma_1  := \gamma_1(p, q, d) := \max \left\{ \frac14, 1 - \frac{p}{d}  \right\} \in (0, 1).
	\end{gather*}
	Then,  there exists a constant $C>0$ such that  for all $T \ge 2$,
	\begin{align*}
		\Ee^\mu\left[\Tcal_p (\mu_T,\mu)\right]
		\leq
		 C\,  T^{-\frac{2}{3} \left(1 - \max\left\{\gamma_1, \, \frac{p}{q}\right\} \right)}&
		\left(\I_{\{\gamma_1 q = p \}} \log T +   \I_{\{\gamma_1 q \ne p \}} \right)\left(\I_{\left\{ p = \frac34 d \right\}} \log T  +   \I_{\left\{ p \ne \frac34 d \right\}} \right).
	\end{align*}
	In particular, if  $q > 4p$, then
	\begin{gather}\label{IP1}
		\Ee^\mu\left[\Tcal_p (\mu_T,\mu)\right]
		\leq
		\begin{cases}
		 C T^{-\frac 1 2}			& \text{if\ \ } p > \frac34 d,\\
		 C T^{-\frac 1 2}\log T	& \text{if\ \ } p = \frac34 d,\\
		 C T^{-\frac {2p} {3d}}	& \text{if\ \ } p < \frac34 d.
		\end{cases}
	\end{gather}
\end{theorem}
The Cauchy--Schwarz inequality shows that \textbf{\upshape(H3)} implies \textbf{\upshape(H2)}. Under the stronger assumption \textbf{\upshape(H3)}, we get better rates.

\begin{theorem}\label{S}
	Let $(X_t)_{t\geq 0}$ be an ergodic Markov process on $\real^d$ with invariant measure $\mu$.  Assume that \textbf{\upshape(H3)} holds, let $p>0$, and assume that $\mu$ has for some $q> p$ a $q$-th moment $\mu(|\cdot|^q)  < \infty$. Define
	\begin{gather*}
		\gamma_2  := \gamma_2(p, q, d) := \max \left\{ \frac12, 1 - \frac{p}{d}  \right\} \in (0, 1).
	\end{gather*}
	Then,  there exists a constant $C>0$ such that  for all $T \ge 2$,
	\begin{equation}\label{S-W}\begin{split}
		\Ee^\mu \left[\Tcal_p (\mu_T,\mu)\right]
		\leq
		 C T^{- \left(1 - \max\{\gamma_2, \, \frac{p}{q}\} \right)}
		&\left(\I_{\{\gamma_2 q = p \}} \log T  +   \I_{\{\gamma_2 q \ne p \}} \right)\left(\I_{\left\{ p = \frac12 d \right\}} \log T  +   \I_{\left\{ p \ne \frac12 d \right\}} \right).
	\end{split}\end{equation}
	In particular, if  $q>2p$, then
	\begin{gather}\label{S1}
	\Ee^\mu \left[ \Tcal_p(\mu_T,\mu)\right]
	\leq
	\begin{cases}
	 C T^{-\frac 1 2}		 & \text{if\ \ } p > \frac{d}{2},\\
	 C T^{-\frac 1 2}\log T& \text{if\ \ } p = \frac{d}{2},\\
	 C T^{-\frac p d}		 & \text{if\ \ } p < \frac{d}{2}.
	\end{cases}
	\end{gather}
\end{theorem}

 Theorems~\ref{D}--\ref{S} characterize the behaviour in expectation.  We can also get almost sure upper bounds for $\Tcal_p(\mu_T,\mu)$:
\begin{theorem}\label{thm3}
	 Let $(X_t)_{t\geq 0}$ be an ergodic Markov process on $\real^d$ with invariant measure $\mu$. Assume that \textbf{\upshape(H2)} holds, let $p>0$, and assume that $\mu$ has for some $q> p$ a $q$-th moment $\mu(|\cdot|^q)  < \infty$. For the rate function
\begin{gather}\label{rate1}
	R_\eta(T) :=
	\begin{cases}
		T^{-\frac{p}{2(p+d)}} (\log T)^\eta 	& \text{if\ \ } p + d < \frac{q}{4}\\
		T^{-\frac{2p}{q}} (\log T)^{\frac32}    & \text{if\ \ } p + d = \frac{q}{4} \\
	T^{-\frac{2p(q-p)}{q(3p+4d)}} (\log T)^\eta & \text{if\ \ } p + d > \frac{q}{4}
	\end{cases} \qquad \text{with a fixed\ \ } \eta>1,
\end{gather}
	one has
	\begin{gather} \label{AS}
		\limsup_{T \to \infty} \frac{\Tcal_p(\mu_T,\mu)}{R_\eta(T)}  < \infty \quad \text{almost surely}.
	\end{gather}
	If instead of \textbf{\upshape(H2)} the stronger condition \textbf{\upshape(H3)} is assumed, we get the same a.s.\ result
\eqref{AS} with the following modified rate function in place of $R_\eta(T)$:
\begin{gather}\label{rate2}
	 \widetilde R_\eta(T) :=
	\begin{cases}
		T^{-\frac{p}{2(p+d)}} (\log T)^\eta & \text{if\ \ } p + d < \frac{q}{2}\\
		T^{-\frac{p}{q}} (\log T)^{\frac32}     & \text{if\ \ } p + d = \frac{q}{2} \\
		T^{-\frac{p(q-p)}{q(p +  2d)}} (\log T)^\eta  & \text{if\ \ } p + d > \frac{q}{2}
	\end{cases} \qquad \text{with a fixed\ \ } \eta>1.
	\end{gather}
\end{theorem}

The remainder of the paper is organized as follows. In Section~\ref{sec-pre} we explain the relations
among the conditions \textbf{(H1)}, \textbf{(H2)} and \textbf{(H3)}, and we provide some useful lemmas needed in the proof of the main results.  Section~\ref{sec-proof-lp} is devoted to the proofs of the $p$-Wasserstein convergence in the mean sense, and Section~\ref{sec-proof-as} is about almost sure convergence. In Section~5 we apply our findings to specific processes, including diffusions and underdamped Langevin dynamics, and compare them with existing works.

\section{Preliminaries}\label{sec-pre}

\subsection{More about the assumptions (H1)--(H3)}\label{subsec-more}

Let us briefly discuss the relations
 among the assumptions \textbf{(H1)}, \textbf{(H2)} and \textbf{(H3)}. In practice, the exponential decay of the variance  $\Var_\mu(f) := \int \left(f-\mu(f)\right)^2\,\dup\mu = \mu\left((f-\mu(f))^2\right)$  along the semigroup is frequently considered. Therefore, we introduce the following auxiliary assumption:
\begin{enumerate}
	\item[\textbf{(H2$\mbox{}^\prime$)}]
	There are constants $C \ge 1$ and $\lambda_{\mathrm{V}}>0$ such that for all $f \in L^2(\mu)$  and $t\geq 0$,
	\begin{align*}
		\Var_\mu(P_t f) \leq C \eup^{-2\lambda_{\mathrm{V}} t } \, \Var_\mu(f).
	\end{align*}
\end{enumerate}
This condition is in-between \textbf{\upshape(H2)} and \textbf{\upshape(H3)}, to wit
\begin{gather*}
	\textbf{(H3)}
	\implies \textbf{(H2$\mbox{}^\prime$)}
	\implies \textbf{(H2)}.
\end{gather*}
Indeed, since  $\Var_\mu(P_tf) = \mu\left((P_tf)^2\right) - \mu\left(P_tf\right)^2 = \mu\left((P_tf)^2\right) - \mu(f)^2$, we have
\begin{gather*}
	\frac{\dup}{\dup t} \Var_\mu(P_t f)
	= 2 \langle \Lcal P_t f,  P_t f \rangle_{L^2(\mu)},
\end{gather*}
so \textbf{\upshape(H3)} yields \textbf{(H2$\mbox{}^\prime$)} with $C = 1$ and $\lambda_{\mathrm{V}} = \lambda_{\mathrm{C}}$.

Now we assume that \textbf{(H2$\mbox{}^\prime$)} holds. For any $f \in \mathscr D(\Lcal)$ with $\mu(f) = 0$ we can use $\partial_t P_t f = P_t \Lcal f$ and $L^2(\mu)\text{-}\lim_{T \to \infty} P_T f = 0$, to see
\begin{gather*}
	f = - \int_{0}^{\infty} P_t \Lcal f \,\dup t.
\end{gather*}
Combining this with \textbf{(H2$\mbox{}^\prime$)}, we obtain
\begin{gather*}
	\| f \|_{L^2(\mu)}
	\leq \int_{0}^{\infty} \| P_t \Lcal f \|_{L^2(\mu)} \,\dup t
	\leq \sqrt{C} \int_{0}^{\infty} \eup^{- \lambda_{\mathrm{V}} t} \,\dup t  \cdot \| \Lcal f \|_{L^2(\mu)}
	 = \frac{\sqrt{C}}{\lambda_{\mathrm{V}}} \| \Lcal f \|_{L^2(\mu)}.
\end{gather*}
Hence, \textbf{\upshape(H2)} holds with $\lambda_{\mathrm{I}} = C^{-\frac12}\lambda_{\mathrm{V}}$.

Because of the Kantorovich duality, cf.\ \cite[Remark 7.5]{V-2003}, \textbf{\upshape(H1)} implies that for any Lipschitz function $f$ on $\real^d$, all $x \in \real^d$ and $t\ge 0$,
\begin{align*}
	|P_tf(x)-\mu(f)|
	\leq  \Wa_1(\delta_x P_t,\mu P_t) [ f ]_{\Lip}
	\leq C  \eup^{-\lambda_{\mathrm{E}} t} \Wa_1(\delta_x,\mu) \, [ f ]_{\Lip},
\end{align*}
where $[ f ]_{\Lip}  := \sup_{x\neq y} |f(x)-f(y)|/|x-y|$ denotes the Lipschitz constant  (or Lipschitz seminorm)  of $f$. Moreover, if $\mu(|\cdot|^q)  < \infty$ for some $q \ge 1$, then $f \in L^q(\mu)$ and
\begin{gather}\label{decay-for-Lip}
		\| P_t f - \mu(f) \|_{L^q(\mu)}
		\leq C_\mu \eup^{-\lambda_{\mathrm{E}} t} [ f ]_{\Lip}
\end{gather}
for some constant $C_\mu > 0$ depending only on $\mu$. If the generator $\Lcal$ is normal, i.e.\ $\Lcal \Lcal^* =  \Lcal^* \Lcal$, and $q = 2$, then the spectral representation of normal operators allows us to deduce \textbf{\upshape(H3)} with $\lambda_{\mathrm{C}} = \lambda_{\mathrm{E}}$ from \eqref{decay-for-Lip}, see \cite[Theorem 4.1.4]{W05}. The same theorem also shows that \textbf{(H2$\mbox{}^\prime$)} implies \textbf{\upshape(H3)}. Summing up, we have
\begin{gather*}
	\text{if $\Lcal$ is normal, then:}\quad
	\textbf{(H1)}\;\&\; \mu(|\cdot |^2) < \infty \implies \textbf{(H3)} \iff \textbf{(H2$\mbox{}^\prime$)}.
\end{gather*}

It is well known that functional inequalities of the type \textbf{(H1)}--\textbf{(H3)} are closely connected with curvature conditions. A classical example is provided by the symmetric diffusion on $\real^d$ under the Bakry-\'Emery curvature condition  $\mathrm{CD}(\rho, \infty)$ for some $\rho > 0$: set $\mu(\dup x) := \eup^{-V(x)}\,\dup x$, where the potential $V \in C^2(\real^d)$ satisfies $\Hess \, V \geq \rho \, \mathds{I}_{d\times d}$ and is such that $\mu$ becomes a probability measure. Consider the diffusion with the generator $\Lcal = \Delta - \nabla V \cdot \nabla$ and the transition operators $P_t = \eup^{t \Lcal}$. For this diffusion, \textbf{(H1)}--\textbf{(H3)} are satisfied, see e.g.\ \cite{BGL14} or \cite{W05}. Recently, \textbf{(H1)} (even with both $\mu$ and $\nu$ being \emph{arbitrary} probability measures) has been extended to diffusions whose potentials are not uniformly convex, see e.g.\ \cite{E11, LW16, W20}. For related functional inequalities under a variable curvature lower bound, see also \cite{CFG}.

For the proof of our main result, we will need Bernstein-type inequalities for Markov processes, which quantitatively describe the concentration behaviour of the processes. The tail inequality under \textbf{(H2)} can be found in \cite[Theorem 3.4]{HL}, whereas the analogue under the weaker assumption \textbf{(H3)} was  obtained in \cite[Theorem 1.1 and Remark 1.2]{L}.
\begin{theorem}\label{tail}
	Let $\nu  = \Lscr_{X_0}$ denote the initial distribution of the ergodic Markov process $(X_t)_{t\geq 0}$ on $\real^d$ with invariant measure $\mu$, and assume that $\nu(\dup x) = h_\nu(x)\,\mu(\dup x)$. Let $f : \real^d \to \real$ be a measurable function such that $\mu(f) = 0$, $\| f \|_\infty \le M$ and $\Var_\mu(f) \leq \sigma^2$ for some $M, \sigma>0$.
\begin{enumerate}[\upshape (i)]
\item \label{Bern-IP}
	Assume that \textbf{\upshape(H2)} holds and $h_\nu \in L^p(\mu)$ for some $p \in (1, \infty]$. Denote by $q:= p/(p-1)$ the conjugate index of $p$. Then for any $T > 0$ and $\delta > 0$,
	\begin{align*}
		\Pp^\nu \left(\left|\frac 1 T \int_0^T f(X_t) \, \dup t\right| \geq \delta\right)
		&\leq 2 \|h_\nu\|_{L^p(\mu)}\exp\left[-\frac{\lambda_{\mathrm{I}} T \delta^2}{4 q M \sqrt{4\sigma^2+\delta^2}}\right]\\
		&\leq 2 \|h_\nu\|_{L^p(\mu)}\exp\left[-\frac{\lambda_{\mathrm{I}} T}{4\sqrt{5} q}  \cdot \min \left\{\frac{\delta^2}{M \sigma}, \frac{\delta}{M}  \right\}  \right].
	\end{align*}
	Consequently,
	\begin{gather} \label{moment2}
		\Ee^\nu \left[\left| \frac1{T} \int_0^T f(X_t) \, \dup t \right| \right]
		\les \| h_\nu \|_{L^p(\mu)} \sqrt{M} \left( \sqrt{\sigma} \, T^{-\frac12} + \sqrt{M} T^{-1} \right).
	\end{gather}

\item\label{Bernstein}
	Assume that \textbf{\upshape(H3)} holds and $h_\nu \in L^2(\mu)$. Then for any $T > 0$ and $\delta > 0$,
	\begin{align*}
		\Pp^\nu\left( \left| \frac 1T \int_{0}^{T} f(X_t) \, \dup t \right| \ge \delta \right)
		&\leq 2  \left\| h_\nu \right\|_{L^2(\mu)} \exp\left[- \frac{\lambda_{\mathrm{C}} T \delta^2}{(\sigma + \sqrt{\sigma^2 + 2 M \delta})^2} \right]\\
		&\leq 2  \left\| h_\nu \right\|_{L^2(\mu)} \exp\left[- \frac{\lambda_{\mathrm{C}} T}{4 + 2\sqrt{3}} \cdot \min \left\{ \frac{\delta^2}{\sigma^2}, \frac{\delta}{ M}  \right\}  \right].
	\end{align*}
	Consequently,
	\begin{gather} \label{moment1}
		\Ee^\nu \left[ \left| \frac1{T} \int_0^T f(X_t) \, \dup t \right| \right]
		\les \| h_\nu \|_{L^2(\mu)} \left( \sigma T^{-\frac12} + M T^{-1} \right).
	\end{gather}
\end{enumerate}
\end{theorem}
The moment bounds \eqref{moment2} and \eqref{moment1}  follow directly from the tail bounds via the layer-cake formula:
\begin{gather*}
	\Ee^\nu \left[ \left| \frac1{T} \int_0^T f(X_t) \, \dup t \right| \right]
	= \int_{0}^{\infty} \Pp^\nu \left(\left|\frac 1 T \int_0^T f(X_t)\,\dup t\right| \ge \delta\right) \dup \delta.
\end{gather*}
 For large $T\gg 1$, the inequalities in \eqref{Bernstein} are better if $\sigma \ll M$.

\subsection{Upper bounds for $\Tcal_p$} \label{2.2}

Let us briefly recall a key tool, due to \cite[Lemma 5 and 6]{FG}, to obtain upper bounds for $\Tcal_p$; see also \cite{DSS, HK} for related bounds. For each $\ell \in \nat$, denote by $\Pcal_\ell$ the natural partition of $(-1,1]^d$ into $2^{d\ell}$ dyadic cubes of side-length $2\cdot 2^{-\ell}$. For $F \in \Pcal_\ell$ and $n\in \nat$, we write $2^n F:= \{2^n x \,:\, x\in F\}$,  and set
\begin{gather*}
	B_0 := (-1,1]^d,
	\quad
	B_n := (-2^n,2^n]^d \setminus (-2^{n-1},2^{n-1}]^d, \quad n\ge 1.
\end{gather*}
\begin{lemma}\label{Wp}
	For every $p>0$, there exists a constant $C_p>0$ such that for all $\nu_0,\nu_1 \in\Pscr(\real^d)$,
	\begin{gather*}
		\Tcal_p (\nu_0, \nu_1)
		\leq C_p \sum_{n\ge 0} 2^{pn}\sum_{\ell\geq 0} 2^{-p\ell}\sum_{F\in\Pcal_\ell} \left| \nu_0(2^n F\cap B_n)-\nu_1(2^n F\cap B_n) \right|.
	\end{gather*}
\end{lemma}
With the help of this lemma, the problem of estimating $\Ee^\mu[\Tcal_p (\mu_T, \mu) ]$ reduces to that of controlling mean differences of the form
$
	\Ee^\mu \left[|\mu_T(A) - \mu(A)| \right],
$
where $A$ is a Borel subset of $\real^d$; this is
why the concentration results from Theorem~\ref{tail} will come in.

\section{Proofs of Theorems \ref{D}, \ref{IP} and  \ref{S}}\label{sec-proof-lp}

\subsection{Proof of Theorem \ref{D}}\label{subsec-proof-1}

In this section, we establish upper bounds for $\Ee^\mu [\Tcal_p(\mu_T, \mu)]$ under $\textbf{(H1)}$. Since $\mu_T$ is singular with respect to $\mu$, we first mollify it by convolving it with smooth densities, which has  the form $\rho_\epsilon(x):=\epsilon^{-d}\rho (x/ \epsilon)$, where $\rho$ is a smooth function with support in $\overline{B_1(0)}$. As usual, $B_r(x)$ denotes the open ball in $\real^d$ centred at $x$ with radius $r > 0$.

Let $\xi$ be a $\real^d$-valued random variable, which is independent of the process $(X_t)_{t \ge 0}$ and has a probability density function \(\rho \in C_c^\infty(\real^d) \) such that $\supp\rho \subset \overline{B_1(0)}$. If $Y\sim \nu$ is a further $\real^d$-valued random variable, which is independent of $\xi$, then for any $\epsilon>0$,
$
	\Lscr_{Y + \epsilon \xi} = \nu * \Lscr_{\epsilon \xi}.
$
 Observe that $\epsilon \xi \sim \rho_\epsilon(x)\dup x$. From the definition of $\Tcal_p$, we see that for all $p>0$,
$$
	\Tcal_p (\nu*\Lscr_{\epsilon \xi},\nu)
	\leq \Ee[|Y + \epsilon\xi- Y|^{p}]=\epsilon^{p}\Ee[|\xi|^{p}]
	\lesssim \epsilon^p.
$$
Combining this  with the triangle inequality for $\Wa_p$ shows
\begin{gather} \label{WP-2}
\begin{split}
	\Tcal_p (\mu_T,\mu)
	&\lesssim \Tcal_p \left(\mu_T,\mu_T*\Lscr_{\epsilon \xi}\right) + \Tcal_p \left(\mu_T*\Lscr_{\epsilon \xi},\mu*\Lscr_{\epsilon \xi}\right) + \Tcal_p \left(\mu*\Lscr_{\epsilon \xi},\mu\right)\\
	&\lesssim \epsilon^{p} + \Tcal_p \left(\mu_T*\Lscr_{\epsilon \xi},\mu*\Lscr_{\epsilon \xi}\right).
\end{split}
\end{gather}
We can now use Lemma \ref{Wp} in order to estimate the second term on the right-hand side.

For any Borel set $A \subseteq \real^d$, it is easily checked that
\begin{gather} \label{pmu}
	\left( \mu_T*\Lscr_{\epsilon \xi} \right) (A)
	= \frac1T \int_{0}^{T} \Pp (X_t + \epsilon \xi \in A) \, \dup t,\,\,\,
	\left( \mu*\Lscr_{\epsilon \xi} \right) (A)
	= \Pp(X+\epsilon \xi \in A),
\end{gather}
where $X$ is a random variable that is independent of $\xi$ with $\Lscr_X = \mu$. If $\Lscr_{X_0} = \mu$, then $\Lscr_{X_t} = \mu$, since $\mu$ is an invariant measure. Consequently,
\begin{gather}\begin{split} \label{Inv}
	&\Ee^\mu \left[\left| \left( \mu_T*\Lscr_{\epsilon \xi} \right) (A) - \left( \mu*\Lscr_{\epsilon \xi} \right)(A) \right|  \right] \\
	&\leq \Ee^\mu \left[ \left( \mu_T*\Lscr_{\epsilon \xi} \right) (A) \right] + \left( \mu*\Lscr_{\epsilon \xi} \right)(A)\\
	&\leq \frac1T \int_{0}^{T} \Ee^\mu \left[ \Pp (X_t + \epsilon \xi \in A) \right] \dup t + \Pp (X+\epsilon \xi \in A)\\
	&= 2 \Pp(X+\epsilon \xi \in A).
\end{split}
\end{gather}
For $\epsilon\in (0,1)$ and any Borel set $A\subset\real^d$, we  define the function $f_{A, \epsilon}: \real^d \to \real$ by
\begin{gather}\label{f-eps}
\begin{split}
	f_{A, \epsilon}(z)
	:=& \left( \delta_z * \Lscr_{\epsilon \xi} \right) (A) - \left( \mu*\Lscr_{\epsilon \xi} \right) (A)\\
	=& \int_A\rho_\epsilon(x-z)\, \dup x -\int_{\real^d} \left( \int_A \rho_\epsilon(x-y)\, \dup x \right) \mu(\dup y),
	\quad z \in \real^d.
\end{split}
\end{gather}
Clearly, $\mu(f_{A, \epsilon})=0$ and $\| f_{A, \epsilon} \|_\infty \le 1$. Furthermore, $f_{A, \epsilon}$ is Lipschitz continuous with
\begin{align} \label{lips}
	[f_{A, \epsilon}]_{\Lip}
	\leq \min \left\{ \epsilon^{-d-1} \vol(A),  2 \vol(B_1(0)) \epsilon^{-1} \right\} \cdot [\rho]_{\Lip},
\end{align}
where $\vol(\cdot)$ is Lebesgue measure. Indeed, since $\supp\rho_\epsilon\subset \overline{B_\epsilon(0)}$, we have for any $z_1, z_2 \in \real^d$,
\begin{align*}
	|f_{A, \epsilon}(z_1) - f_{A, \epsilon}(z_2)|
	&= \left| \int_{A \cap B_\epsilon(z_1)}\rho_\epsilon(x-z_1) \,\dup x - \int_{A \cap B_\epsilon(z_2)}\rho_\epsilon(x-z_2)\,\dup x  \right|\\
	&\leq \int_{A \cap (B_\epsilon(z_1) \cup B_\epsilon(z_2))} \left|\rho_\epsilon(x-z_1) - \rho_\epsilon(x-z_2)\right| \dup x\\
	&\leq \vol (A \cap (B_\epsilon(z_1) \cup B_\epsilon(z_2))) \cdot [\rho_\epsilon]_{\Lip} |z_1 - z_2|,
\end{align*}
which gives \eqref{lips}. By the Markov property of $(X_t)_{t\ge0}$ and the invariance of $\mu$ under the operators $P_t$ for all $t>0$, we deduce that
\begin{align*}
	&\Ee^\mu \left[\left| \left( \mu_T*\Lscr_{\epsilon \xi} \right) (A)- \left( \mu*\Lscr_{\epsilon \xi} \right)(A) \right|^2  \right] \\
	&= \Ee^\mu \left[\left|\frac 1 T \int_0^T f_{A, \epsilon} (X_t)\,\dup t\right|^2\right]
	= \Ee^\mu \left[\frac 1 {T^2} \int_0^T \int_0^T f_{A, \epsilon} (X_t) \,f_{A, \epsilon} (X_s)\,\dup t\,\dup s\right]\\
	&= \frac 2 {T^2} \int_0^T\int_{s}^T\Ee^{\mu} \left[f_{A, \epsilon}(X_{s}) \, f_{A, \epsilon}(X_{t}) \right]\dup t\,\dup {s}
	=\frac 2 {T^2}\int_0^T\int_{s}^T \mu \left( P_{s} \left( f_{A, \epsilon} \, P_{t-s}f_{A, \epsilon} \right) \right)\dup t\,\dup {s}\\
	&= \frac 2 {T^2}\int_0^T\int_{s}^T\mu \left( f_{A, \epsilon} \, P_{t-s}f_{A, \epsilon} \right)\dup t\,\dup {s}
	\leq \frac 2 {T^2}\int_0^T\int_{s}^T \left\| P_{t- s} f_{A, \epsilon} \right\|_{L^1(\mu)}\,\dup t\,\dup {s}.
\end{align*}
Assuming \textbf{(H1)} and $\mu(|\cdot|) < \infty$, we get from \eqref{decay-for-Lip} that for any $t > 0$,
\begin{align*}
	\left\| P_t f_{A, \epsilon} \right\|_{L^1 (\mu)}
	=  \left\| P_t f_{A, \epsilon} - \mu(f_{A, \epsilon})\right\|_{L^1 (\mu)}
	\les \eup^{-\lambda_{\mathrm{E}} t}[f_{A, \epsilon}]_{\Lip}.
\end{align*}
Therefore, using the Cauchy-Schwarz inequality along with \eqref{lips}, we have
\begin{gather}\label{DM}
	\Ee^\mu \left[\left| \left( \mu_T*\Lscr_{\epsilon \xi} \right) (A)- \left( \mu*\Lscr_{\epsilon \xi} \right)(A) \right|\right]
	\lesssim  \frac{1\wedge\sqrt{\epsilon^{-d}\vol(A)}}{\sqrt{T\epsilon}}.
\end{gather}
Combining \eqref{Inv} with \eqref{DM}, we see for any Borel set $A \subseteq \real^d$ and $\epsilon \in (0, 1)$,
\begin{align*}
	\Ee^\mu \left[\left| \left( \mu_T*\Lscr_{\epsilon \xi} \right) (A)- \left( \mu*\Lscr_{\epsilon \xi} \right)(A) \right|\right]
	\lesssim
	\min\left\{\Pp(X+\epsilon \xi \in A), \frac{1\wedge\sqrt{\epsilon^{-d} \vol(A)}}{\sqrt{T\epsilon}} \right\},
\end{align*}
where we recall that $X\sim\mu$ is independent of $\xi$.

We will now use the notation introduced in Section~\ref{2.2}. For every $n, \ell \in \nat$ and $\epsilon \in (0, 1)$,
\begin{align*}
	&\sum_{F \in \mathcal P_\ell} \Ee^\mu \left[\left| \left( \mu_T*\Lscr_{\epsilon \xi} \right) (2^n F\cap B_n)- \left( \mu*\Lscr_{\epsilon \xi} \right)(2^n F\cap B_n) \right|\right]\\
	&\lesssim  \min \left\{ \sum_{F \in \mathcal P_\ell} \Pp(X + \epsilon \xi \in 2^n F\cap B_n), T^{-\frac12} \epsilon^{-\frac12} \sum_{F \in \mathcal P_\ell} \big(1\wedge \sqrt{\epsilon^{-d} \vol(2^n F\cap B_n)}\big)  \right\}\\
	&\les \min\left\{ \Pp(X + \epsilon \xi \in B_n), T^{-\frac12} \epsilon^{-\frac12} \sum_{F \in \mathcal P_\ell} \big(1\wedge \sqrt{\epsilon^{-d} \vol(2^n F\cap B_n)}\big) \right\}.
\end{align*}
Since $|\xi| \le 1$ a.s., we find for the $\epsilon$-enlarged set  $B_n^\epsilon :=  \left\{x\in\real^d \,:\, |x-y| < \epsilon \text{\ for some\ } y\in B_n\right\}$
that
$
	\Pp(X + \epsilon \xi \in B_n) \leq  \Pp(X\in B_n^\epsilon) =  \mu(B_n^\epsilon),
$
Since $\mu$ has a finite $q$-th moment  $\mu(|\cdot|^q)<\infty$, we can use the Markov inequality to get for all $\epsilon\in (0, 1)$ and $n \in \nat$,
\begin{gather*}
	\mu(B_n^\epsilon)
	= \Pp(X\in B_n^\epsilon)
	\lesssim 2^{-qn},
\end{gather*}
and, therefore, $\Pp(X + \epsilon \xi \in B_n) \les 2^{-qn}$.

Using that $\#\Pcal_\ell=2^{d\ell}$  and $\vol(2^n F) =  2^d\cdot  2^{(n - \ell)d}$, we obtain
\begin{gather*}
	\sum_{F \in \mathcal P_\ell} 1\wedge \sqrt{\epsilon^{-d} \vol(2^n F\cap B_n)}
	\les  \left[ (2^{n-\ell} \epsilon^{-1})^{\frac{d}2} \wedge  1   \right] \cdot 2^{d\ell}.
\end{gather*}
 Combining all estimates from above, we see that
\begin{gather}\label{Wp-F}
\begin{split}
	\sum_{F \in \mathcal P_\ell} \Ee^\mu &\left[\left| \left( \mu_T*\Lscr_{\epsilon \xi} \right) (2^n F\cap B_n)- \left( \mu*\Lscr_{\epsilon \xi} \right)(2^n F\cap B_n) \right|\right]\\
	&\lesssim  \min \left\{ 2^{-qn},  T^{-\frac12} \epsilon^{-\frac12} \left[ (2^{n-\ell} \epsilon^{-1})^{\frac{d}2} \wedge  1   \right] \cdot 2^{d \ell} \right\}.
\end{split}
\end{gather}
For fixed $p>0$ and $q > \max\{p, 1\}$ and every $n \in \nat$, $T \ge 2$, $\epsilon \in \left(0, \frac12\right)$, we define
\begin{gather*}
	\Ical_{p, q}(n, T, \epsilon)
	:= \sum_{\ell \ge 0} 2^{-p\ell} \min \left\{ 2^{-qn},  T^{-\frac12} \epsilon^{-\frac12} \left[ (2^{n-\ell} \epsilon^{-1})^{\frac{d}2} \wedge  1   \right] \cdot 2^{d \ell} \right\}.
\end{gather*}
Set $k  := \log_2(\epsilon^{-1})$ and $k' := \log_2(T^{\frac12} \epsilon^{\frac12})$. Then
\begin{gather*}
	\Ical_{p, q}(n, T, \epsilon)
	= \sum_{\ell \ge 0} 2^{-p\ell + \min \{ -qn, \, -k'+ d \ell -\frac{d}2 (\ell - n - k)_+   \}}.
\end{gather*}
Consequently,
\begin{gather*}
	\Ical_{p, q}(n, T, \epsilon)
	\les \sum_{0 \le \ell < \ell_0} 2^{-p \ell -k'+ d \ell -\frac12 d(\ell - n - k)_+} + 2^{-qn} \sum_{\ell \ge \ell_0} 2^{-p\ell},
\end{gather*}
where
\begin{gather*}
	\ell_0
	:=
	\begin{cases}
		0   				& \text{if\ \ } T^{\frac 12} \epsilon^{\frac 12} < 2^{qn}, \\
		\frac{k'-qn}{d}   	& \text{if\ \ } T^{\frac12} \epsilon^{\frac12} \geq 2^{qn} \;\;\&\;\; T^{\frac12} \epsilon^{d + \frac12} \leq 2^{(q + d)n},\\
		\frac{2(k'-q n)}{d} - n - k  & \text{if\ \ } T^{\frac12} \epsilon^{d + \frac12} > 2^{(q + d)n}. \\
	\end{cases}
\end{gather*}
To keep notation simple, we define for $T \ge 2$ and $\epsilon \in (0, \frac12)$,
\begin{gather*}
	N_1 := \frac{\log_2\left(T^{\frac12} \epsilon^{d + \frac12}\right)}{q+d}
	\quad\text{and}\quad
	N_2 := \frac{\log_2\left(T^{\frac12} \epsilon^{\frac12}\right)}{q}.
\end{gather*}
It is not hard to see that $N_2 - N_1 = O\left(\log_2 \left(T\epsilon^{-(2q-1)}\right)\right)$.  Now we consider three cases. Below, we use several times the fact that a geometric series can be estimated by its leading term times a constant.

\noindent\textbf{Case (i) $p \ge d$}: A direct computation leads to
\begin{gather*}
	\Ical_{p, q}(n, T, \epsilon)
	\les 2^{-qn} \I_{\{ n > N_2 \}}
	+
	\begin{cases}
 	(T\epsilon)^{-\frac12} \I_{\{ n \le N_2 \}}    & \text{if\ \ } p > d, \\
    (T\epsilon)^{-\frac12}[\ell_0 \wedge (n + \log_2(\epsilon^{-1}))] \I_{\{ n \le N_2 \}}  & \text{if\ \ } p = d.
   \end{cases}
\end{gather*}
Therefore,
\begin{gather}
	\sum_{n \ge 0} 2^{pn} \Ical_{p, q}(n, T, \epsilon)
	\les
	\begin{cases}
 	(T\epsilon)^{-\frac12(1 - \frac{p}{q})}   & \text{if\ \ } p > d, \\
    (T\epsilon)^{-\frac12(1 - \frac{d}{q})} \log(T \epsilon^{-(2q - 1)})  & \text{if\ \ } p = d.
 	\end{cases}
\end{gather}

\noindent\textbf{Case (ii) $p \in [d/2, d)$}: In this case we have
\begin{align*}
	&\Ical_{p, q}(n, T, \epsilon) \les 2^{-qn} \I_{\{ n > N_2 \}} \\
	&\mbox{}  +
	\begin{cases}
  		(T \epsilon)^{-\frac{p}{2d}} 2^{-(1 - \frac{p}{d})qn} \I_{\{ N_1 \le n \le N_2 \}} + T^{-\frac12} \epsilon^{-(d - p)- \frac12} 2^{(d - p)n} \I_{\{ n < N_1 \}}
  		&\text{if\ \ } p \in \left(\tfrac d2, d\right), \\
  		(T \epsilon)^{-\frac14} 2^{-\frac{qn}{2}} \I_{\{ N_1 \le n \le N_2 \}} +  T^{-\frac12} \epsilon^{-\frac{d}{2} - \frac12} 2^{\frac{dn}{2}} ( \ell_0 - n - \log_2(\epsilon^{-1}))  \I_{\{ n < N_1 \}}
  		& \text{if\ \ } p  = \tfrac d2.
   \end{cases}
\end{align*}
Therefore, if $p \in (d/2, d)$,
\begin{gather}
	\sum_{n \ge 0} 2^{pn} \Ical_{p, q}(n, T, \epsilon)
	\les
	\begin{cases}
		T^{-\frac{q}{2(q+d)}} \epsilon^{p - \frac{(2d+1)q}{2(q+d)}}
		&\text{if\ \ } q > \frac{dp}{d - p}, \\
 		(T\epsilon)^{-\frac{p}{2d}} \log(T\epsilon^{-(2q-1)})
 		&\text{if\ \ } q = \frac{dp}{d - p}, \\
    	(T\epsilon)^{-\frac12(1 - \frac{p}{q})}
    	&\text{if\ \ } q < \frac{dp}{d - p}.
 	\end{cases}
\end{gather}
If $p = d/2$, we use that for $n < N_1$,
\begin{gather*}
	\ell_0 - n - \log_2(\epsilon^{-1})
	\leq \frac{2k'-dk}{d} - \log_2(\epsilon^{-1})
	= O\left(\log\left(T\epsilon^{2d+1}\right)\right),
\end{gather*}
and so
\begin{gather}
	\sum_{n \ge 0} 2^{pn} \Ical_{p, q}(n, T, \epsilon)
	\les
	\begin{cases}
 		T^{-\frac{q}{2(q+d)}} \epsilon^{\frac{d}{2} - \frac{(2d+1)q}{2(q+d)}} \log(T\epsilon^{2d+1})
 		&\text{if\ \ } q > d, \\
		(T\epsilon)^{-\frac14} \log(T\epsilon^{-(2q-1)})
		&\text{if\ \ } q = d, \\
    	(T\epsilon)^{-(\frac12 - \frac{d}{4q})}
    	& \text{if\ \ } q < d.
	\end{cases}
\end{gather}

\noindent\textbf{Case (iii) $p < d/2$}: In this case we have
\begin{align*}
	\Ical_{p, q}(n, T, \epsilon)
	\les &2^{-qn} \I_{\{n > N_2  \}} +  (T \epsilon)^{-\frac{p}{2d}} 2^{-(1 - \frac{p}{d})qn} \I_{\{ N_1 \le n \le N_2 \}} \\
	&+ T^{-\frac{p}{d}} \epsilon^{-p -\frac{p}{d} } 2^{[p - (1 - \frac{2p}{d})q]n} \I_{\{ n < N_1   \}}.
\end{align*}
Therefore,
\begin{gather*}
	\sum_{n \ge 0} 2^{pn} \Ical_{p, q}(n, T, \epsilon)
	\les
	\begin{cases}
  		T^{-\frac{p}{d}} \epsilon^{-p - \frac{p}{d}}  & \text{if\ \ } q> \frac{2dp}{d - 2p}, \\
  		T^{-\frac{p}{d}} \epsilon^{-p - \frac{p}{d}} \log(T \epsilon^{2d + 1})   & \text{if\ \ } q = \frac{2dp}{d - 2p}, \\
  		T^{-\frac{q}{2(q+d)}} \epsilon^{p - \frac{q(2d + 1)}{2(q + d)}} & \text{if\ \ } q  \in \left(\frac{dp}{d - p}, \frac{2dp}{d - 2p}\right), \\
  		(T\epsilon)^{-\frac{p}{2d}}  \log(T\epsilon^{-(2q-1)})  & \text{if\ \ } q = \frac{dp}{d - p}, \\
  		(T\epsilon)^{-\frac12(1 - \frac{p}{q})}  & \text{if\ \ } q < \frac{dp}{d - p}.
   \end{cases}
\end{gather*}

Combining Lemma \ref{Wp} with \eqref{WP-2} and \eqref{Wp-F}, we obtain for any $T \ge 2$ and $\epsilon \in (0, 1)$ that
\begin{gather*}
	\Ee^\mu [\Tcal_p (\mu_T, \mu)]
	\les \epsilon^p + \sum_{n \ge 0} 2^{pn} \, \Ical_{p, q}(n, T, \epsilon).
\end{gather*}
Optimizing in $\epsilon \in (0, 1)$ yields, if $p \in (0, d)$,  then
\begin{align*}
	\Ee^\mu [\Tcal_p (\mu_T, \mu)]
	\lesssim
	\begin{cases}
  		T^{-\frac{p}{2d+1}}   & \text{if\ \ } q > \frac{dp}{d - p}, \\
   		T^{-\frac{p}{2d+1}} (\log T)^{\frac{2d}{2d+1}}  & \text{if\ \ } q = \frac{dp}{d - p}, \\
  		T^{-\frac{q- p}{2q + (q/p) - 1}}  & \text{if\ \ } q < \frac{dp}{d - p};
   \end{cases}
\end{align*}
if $p \ge d$, then we get in this way
\begin{align*}
	\Ee^\mu [\Tcal_p (\mu_T, \mu)]
	\lesssim
	\begin{cases}
  		T^{-\frac{q- p}{2q + (q/p) - 1}}   & \text{if\ \ } p > d, \\
  		T^{-\frac{q- d}{2q + (q/d) - 1}} (\log T)^{\frac{2q}{2q + (q/d) - 1}}   & \text{if\ \ } p = d.
   \end{cases}
\end{align*}
This completes the proof of Theorem \ref{D}.

\subsection{Proofs of Theorems \ref{IP} and \ref{S}}\label{subsec-proof-23}

We will prove Theorem~\ref{IP} and Theorem~\ref{S} with a unified argument based on the Bernstein-type inequalities from Theorem \ref{tail}.

Because $\mu$ is an invariant measure of the semigroup $(P_t)_{t\ge 0}$, we obtain that for every Borel set $A\subset\real^d$
\begin{equation}\label{I}\begin{split}
	\Ee^\mu \left[ |\mu_T(A)-\mu(A)| \right]
	\leq &\Ee^\mu[\mu_T(A)] + \mu(A)\\
	= &\frac 1 T\int_0^T \Pp^\mu(X_t \in A) \,\dup t +\mu(A)=2\mu(A).\end{split}
\end{equation}
 Define the function $f_A: \real^d \to \real$ by
\begin{gather*}
	f_A(z) := \I_A(z) - \mu(A),\quad z \in \real^d.
\end{gather*}
Clearly, $\mu(f_A)=0$ and $\|f_A \|_\infty \leq 1$. A simple computation gives
\begin{gather*}
	\Var_\mu (f_A)
	= \|f_A\|_{L^2(\mu)}^2 = \mu(A)(1-\mu(A))
	\leq \mu(A).
\end{gather*}
Together with
\begin{gather*}
	\Ee^\mu \left[|\mu_T(A)-\mu(A)| \right]
	= \Ee^\mu \left[ \left| \frac1{T} \int_0^T f_A(X_t)\,\dup t \right| \right],
\end{gather*}
we deduce from \eqref{moment2} and \eqref{moment1} that, under \textbf{\upshape(H2)},
\begin{gather*}
	\Ee^\mu \left[ |\mu_T(A)-\mu(A)| \right]
	\lesssim T^{-\frac 1 2} \mu(A)^{\frac 1 4} + T^{-1};
\end{gather*}
while under \textbf{\upshape(H3)},
\begin{gather*}
	\Ee^\mu \left[ |\mu_T(A)-\mu(A)| \right]
	\lesssim T^{-\frac 1 2} \mu(A)^{\frac 1 2} + T^{-1}.
\end{gather*}
Combining this with \eqref{I}, we obtain
\begin{gather*}
	\Ee^\mu \left[ |\mu_T(A)-\mu(A)| \right]
	\lesssim
	\begin{cases}
		\min\left\{\mu(A),\, T^{-\frac12}\mu(A)^{\frac14}\right\} & \text{under\ \ } \textbf{(H2)},\\[4pt]
		\min\left\{\mu(A),\, T^{-\frac12}\mu(A)^{\frac12}\right\} & \text{under\ \ } \textbf{(H3)},
	\end{cases}
\end{gather*}
where we use the fact that $0\leq \mu(A) \leq 1$. This, together with Lemma \ref{Wp}, yields
\begin{gather} \label{Wp-I}
	\Ee^\mu \left[\Tcal_p (\mu_T,\mu)\right]
	\lesssim \sum_{n\ge 0} 2^{pn} \sum_{\ell \ge 0} 2^{-p\ell} \sum_{F\in\Pcal_\ell}
	\min\left\{\mu(2^n F \cap B_n),\: T^{-\frac 1 2} \mu(2^n F \cap B_n)^\alpha \right\},
\end{gather}
where $\alpha = 1/4$ (respectively, $1/2$) under \textbf{(H2)} (respectively, \textbf{(H3)}).

 For every $\alpha \in (0, 1)$ and $T \ge 2$, define
\begin{gather*}
	\Kcal_\alpha (T)
	:= \sum_{n\ge 0} 2^{pn} \sum_{\ell\ge 0}2^{-p\ell} \sum_{F\in\Pcal_\ell}
	\min\left\{\mu(2^n F\cap B_n),\: T^{-\frac12} \mu(2^n F\cap B_n)^{\alpha} \right\}.
\end{gather*}

The following lemma gives an upper bound for  $\Kcal_{\alpha}(T)$.
\begin{lemma} \label{alpha}
	Let $\Kcal_\alpha(T)$, $\alpha$, $p$ and $q$ as above, and define
	\begin{gather*}
		\gamma := \gamma(\alpha, p, d) := \max \left\{ \alpha, 1 - \frac{p}{d}   \right\} \in (0, 1).
	\end{gather*}
	Then,  for all $T>2$,
	\begin{gather*}
		\Kcal_{\alpha}(T)
		\lesssim T^{-\frac{1 - \max\{\gamma, \, p/q\}}{2(1 - \alpha)}} \cdot \left(\I_{\{\gamma q = p \}} \log T  +   \I_{\{\gamma q \ne p \}} \right) \cdot \left(\I_{\{ p = (1 - \alpha)d \}} \log T  +   \I_{\{ p \ne (1 - \alpha)d \}} \right).
	\end{gather*}
\end{lemma}

\begin{proof}
First, for each $n, \ell \in \nat$,
\begin{align*}
	\sum_{F\in\Pcal_\ell} &\min\left\{\mu(2^n F\cap B_n), T^{-\frac12} \mu(2^n F\cap B_n)^{\alpha} \right\}\\
	&\le \min \left\{ \sum_{F\in\Pcal_\ell} \mu(2^n F\cap B_n), \sum_{F\in\Pcal_\ell} T^{-\frac12} \mu(2^n F\cap B_n)^{\alpha} \right\}\\
	&= \min \left\{ \mu(B_n), \sum_{F\in\Pcal_\ell} T^{-\frac12} \mu(2^n F\cap B_n)^{\alpha} \right\}.
\end{align*}
H\"older's inequality shows
\begin{gather*}
	\sum_{F\in\Pcal_\ell}\mu(2^n F\cap B_n)^{\alpha}
	\leq (\#\Pcal_\ell)^{1-\alpha} \left(\sum_{F\in \Pcal_\ell}\mu(2^n F\cap B_n)\right)^{\alpha}
	=  2^{ (1-\alpha) d  \ell} \mu(B_n)^{\alpha}.
\end{gather*}
Hence, we obtain
\begin{gather*}
	\Kcal_{\alpha}(T)
	\leq \sum_{n\ge 0} 2^{pn} \sum_{\ell \ge 0}2^{-p\ell} \min\left\{\mu( B_n), T^{-\frac12} 2^{ (1-\alpha)d\ell} \mu(B_n)^{\alpha} \right\}.
\end{gather*}
Since  $\mu(|\cdot|^q)<\infty$, we can use the Markov inequality to get that for all $n \in \nat$
\begin{gather*}
	\mu(B_n) \les 2^{-qn}.
\end{gather*}
Therefore,
$$
	\Kcal_{\alpha}(T)
	\lesssim \sum_{n\ge 0} 2^{-(q - p)n} \sum_{\ell \ge 0}2^{-p\ell} \min\left\{1,  T^{-\frac12}  2^{ (1-\alpha) (d \ell + q n) } \right\}.
$$
Set $L := \frac1{d} \left(\frac{1}{2(1 - \alpha)}\log_2(T) - q n   \right)$. If $L < 0$, i.e.,\ if $n > \frac{ \log_2(T)}{2(1 - \alpha)q}$, then
\begin{gather*}
	\sum_{\ell \ge 0} 2^{-p\ell} \min\left\{1,  T^{-\frac12}  2^{ (1-\alpha) (d \ell + q n) } \right\}
	\lesssim 1.
\end{gather*}
If $L \geq 0$, i.e.,\ if $n \leq \frac{\log_2(T)}{2(1 - \alpha)q}$, then
\begin{align*}
	\sum_{\ell \ge 0} 2^{-p\ell} \min\left\{1,  T^{-\frac12}  2^{ (1-\alpha) (d \ell + q n) } \right\}
	&= T^{-\frac12} 2^{(1 - \alpha)qn}\sum_{0\le \ell \le L}2^{[(1 - \alpha) d - p]\ell} + \sum_{\ell > L} 2^{-p\ell}\\
	& \lesssim
	\begin{cases}
  		T^{-\frac12} 2^{(1 - \alpha)qn}   				& \text{if\ \ } (1 - \alpha) d < p, \\
  		T^{-\frac12} (\log T) 2^{(1 - \alpha)qn}     	& \text{if\ \ } (1 - \alpha) d = p, \\
  		T^{-\frac{p}{2(1 - \alpha) d}} 2^{\frac{pq}{d}n}& \text{if\ \ } (1 - \alpha) d > p.
   \end{cases}
\end{align*}
We distinguish
among three cases.

\noindent
\textbf{Case (i): $(1 - \alpha) d < p$.} We have
\begin{align*}
	\Kcal_{\alpha}(T)
	\lesssim T^{-\frac12} \sum_{0 \le n \le \frac{\log_2(T)}{2(1 - \alpha)q}} 2^{(p - \alpha q)n} + \sum_{n > \frac{\log_2(T)}{2(1 - \alpha)q}} 2^{-(q -p)n}
	\les
	\begin{cases}
  		T^{-\frac12}   			& \text{if\ \ } \alpha q > p, \\
  		T^{-\frac12} \log T     & \text{if\ \ } \alpha q = p, \\
  		T^{-\frac{q - p}{2(1- \alpha)q}}  & \text{if\ \ } \alpha q < p.
   \end{cases}
\end{align*}

\noindent
\textbf{Case (ii): $(1 - \alpha) d = p$.} We have
\begin{align*}
	\Kcal_{\alpha}(T)
	& \lesssim T^{-\frac12} (\log T) \sum_{0 \le n \le \frac{\log_2(T)}{2(1 - \alpha)q}} 2^{(p - \alpha q)n} + \sum_{n > \frac{\log_2(T)}{2(1 - \alpha)q}} 2^{-(q -p)n}\\
	& \les
	\begin{cases}
  		T^{-\frac12} \log T  			& \text{if\ \ } \alpha q > p, \\
  		T^{-\frac12} (\log T)^2      	& \text{if\ \ } \alpha q = p, \\
  		T^{-\frac{q - p}{2(1- \alpha)q}} \log T  & \text{if\ \ } \alpha q < p.
   \end{cases}
\end{align*}

\noindent
\textbf{Case (iii): $(1 - \alpha) d > p$.} We have
\begin{align*}
	\Kcal_{\alpha}(T)
	& \lesssim T^{-\frac{p}{2(1 - \alpha) d}} \sum_{0 \le n \le \frac{\log_2(T)}{2(1 - \alpha)q}} 2^{[p - (1 - \frac{p}{d})q]n} + \sum_{n > \frac{ \log_2(T)}{2(1 - \alpha)q}} 2^{-(q -p)n}\\
	& \les
	\begin{cases}
  		T^{-\frac{p}{2(1 - \alpha) d}}   		& \text{if\ \ } q > \frac{dp}{d - p}, \\
  		T^{-\frac{p}{2(1 - \alpha) d}} \log T   & \text{if\ \ } q = \frac{dp}{d - p}, \\
  		T^{-\frac{q - p}{2(1- \alpha)q}}  		& \text{if\ \ } q < \frac{dp}{d - p}.
   \end{cases}
\end{align*}
The assertion of Lemma~\ref{alpha} follows from these three cases.
\end{proof}

Finally, Theorem~\ref{IP} (resp.\ Theorem~\ref{S}) follows from \eqref{Wp-I} and Lemma \ref{alpha} with $\alpha = 1/4$ assuming \textbf{(H2)} (resp.\ $\alpha = 1/2$ assuming \textbf{(H3)}).

\section{Proof of Theorem \ref{thm3}}\label{sec-proof-as}

We continue to use the notation introduced in the proof of Theorem~\ref{D}. We use \eqref{WP-2} and Lemma \ref{Wp} with $T \ge 2$ and $\epsilon = \epsilon(T) :=  \entier{T}^{-\theta} \simeq T^{-\theta}$. The parameter $\theta$ will be chosen in the course of the proof. Then,
\begin{gather}\label{C}\begin{split}
	 \Tcal_p(\mu_T,\mu)
	\lesssim &T^{-p\theta}
	+ \sum_{n \ge 0} 2^{pn} \sum_{\ell \ge 0} 2^{-p\ell} \sum_{F \in \mathcal P_\ell}
	\left| \left( \mu_T*\Lscr_{\epsilon \xi} \right) (2^n F\cap B_n)- \left( \mu*\Lscr_{\epsilon \xi} \right)(2^n F\cap B_n) \right|.\end{split}
\end{gather}
Let $N = N(T) := \entier{\kappa  \log_2 \entier{T}}$, where the parameter $\kappa > 0$ which will be chosen later on.  In particular, $2^{N} \simeq T^\kappa$. Set
\begin{equation}
\label{J1T}
	\Jcal_1(T)
	:= \sum_{0\le n \le N } 2^{pn} \sum_{\ell \ge 0} 2^{-p\ell} \sum_{F \in \mathcal P_\ell} \left| \left( \mu_T*\Lscr_{\epsilon \xi} \right) (2^n F\cap B_n)- \left( \mu*\Lscr_{\epsilon \xi} \right)(2^n F\cap B_n) \right|\end{equation} and
\begin{equation}\label{J2T}
	\Jcal_2(T)
	:= \sum_{n > N } 2^{pn} \sum_{\ell \ge 0} 2^{-p\ell} \sum_{F \in \mathcal P_\ell} \left| \left( \mu_T*\Lscr_{\epsilon \xi} \right) (2^n F\cap B_n)- \left( \mu*\Lscr_{\epsilon \xi} \right)(2^n F\cap B_n) \right|.
\end{equation}
We are going to bound  $\Jcal_1(T)$  and  $\Jcal_2(T)$  separately.

\noindent\textbf{Estimating  $\Jcal_1(T)$:}
For every $\epsilon>0$ and $x \in \real^d$, define
\begin{gather*}
	\phi_\epsilon(x, z)
	:= \rho_{\epsilon}(x- z) - \mu(\rho_\epsilon(x- \cdot)), \quad z \in \real^d.
\end{gather*}
By \eqref{f-eps}, we have for every Borel set $A\subseteq \real^d$
\begin{align*}
	\left| \left( \mu_T*\Lscr_{\epsilon \xi} \right) (A)- \left( \mu*\Lscr_{\epsilon \xi} \right)(A) \right|
	&= \left| \int_A  \frac 1 T\int_0^T\phi_\epsilon(x,X_t) \,\dup t  \, \dup x \right|\\
	&\leq \int_A \left| \frac 1 T\int_0^T\phi_\epsilon(x,X_t)\,\dup t \right|  \dup x.
\end{align*}
Using that $\bigcup_{F\in\Pcal_\ell}(2^n F\cap B_n) = B_n$ is a union of disjoint sets, we have
\begin{gather}\label{I1}\begin{split}
	\Jcal_1(T)
	&\leq \sum_{n=0}^{N} 2^{pn}\sum_{\ell \ge 0} 2^{-p\ell}\sum_{F\in\Pcal_\ell} \int_{2^n F \cap  B_n}
	\left|\frac 1 T\int_0^T\phi_\epsilon(x,X_t)\,\dup t\right| \dup x\\
	& =  (1 - 2^{-p})^{-1} \sum_{n=0}^{N} 2^{pn}\int_{B_n}\left|\frac 1 T\int_0^T\phi_\epsilon(x,X_t)\,\dup t\right|\dup x.
\end{split}\end{gather}
Define
\begin{gather*}
	\Phi(T) := \sum_{n=0}^{N} 2^{pn}\int_{B_n}\left|\frac 1 T\int_0^T\phi_\epsilon(x,X_t)\,\dup t\right|\dup x.
\end{gather*}
We use Bernstein's inequality (Theorem \ref{tail}\eqref{Bern-IP}) to estimate $\Phi(T)$. We subdivide the cube $[-2^{N}, 2^{N}]^d$ into disjoint smaller cubes with   side-length $h = h(T)$,  where  $h$ is the reciprocal of an integer. The exact value of $h$ will be chosen later.  Let $\mathcal Q_n$ denote the collection of all cubes from the above partition that are contained in $B_n$; clearly  $\# \mathcal Q_n = \vol (B_n) / h^d$.  Set $\mathcal Q :=\cup_n \mathcal Q_n$. Denote by $x_Q\in\real^d$ the centre of the cube $Q \in \mathcal Q$. Then
\begin{gather} \label{I1'}
\begin{split}
	\Phi(T)
	\leq &\sum_{n=0}^{N} 2^{pn} \sum_{Q \in \mathcal Q_n}\int_{Q} \left|\frac 1 T \int_0^T \phi_\epsilon(x_Q, X_t)\,\dup t\right|\dup x\\
	&+ \sum_{n=0}^{N} 2^{pn} \sum_{Q \in \mathcal Q_n}\int_{Q} \left|\frac 1 T\int_0^T\left( \phi_\epsilon(x, X_t) - \phi_\epsilon(x_Q, X_t) \right) \dup t\right| \dup x\\
	=&: \Phi_1 (T) + \Phi_2 (T).
\end{split}
\end{gather}
We begin with the estimate for $\Phi_2(T)$. Since for fixed $z \in \real^d$,
\begin{gather*}
	|\phi_\epsilon(x, z) - \phi_\epsilon(x', z)|
	\leq 2 \|\nabla \rho_\epsilon \|_\infty |x - x'|
	= 2 \|\nabla \rho \|_\infty \epsilon^{-(d+1)} |x - x'|,\quad x,x'\in \real^d,
\end{gather*}
we have
\begin{gather} \label{J2}
\begin{split}
	\Phi_2(T)
	&\leq 2 \|\nabla \rho \|_\infty \cdot \sum_{n=0}^{N} 2^{pn} \sum_{Q \in \mathcal Q_n} \int_{Q} \epsilon^{-(d+1)} |x - x_Q| \,\dup x
	\les \epsilon^{-(d+1)} h \sum_{n=0}^{N} 2^{(p+d)n}\\
	&\les \epsilon^{-(d+1)} h  2^{(p+d)N} \les h  T^{(p+d)\kappa  + (d+1)\theta}.
\end{split}
\end{gather}
We will now estimate $\Phi_1(T)$. From the definition of $\phi_{\epsilon}$, we see that for any $x\in\real^d$ and $\epsilon \in (0, 1)$,
\begin{gather*}
	\mu(\phi_{\epsilon}(x,\cdot))=0,
	\quad
	\|\phi_{\epsilon}\|_\infty\le 2\|\rho_\epsilon\|_\infty = 2\|\rho\|_\infty \epsilon^{-d}
\end{gather*}
and
\begin{align*}
	\Var_\mu(\phi_{\epsilon}(x,\cdot))
	& =  \| \phi_{\epsilon}(x,\cdot) \|_{L^2(\mu)}^2 \le \|\rho_\epsilon\|_\infty^2 \cdot \mu(B_\epsilon(x))\leq C_1 (1 + |x|)^{-q}  \epsilon^{-2d},
\end{align*}
where in the last line we used  $\mu(|\cdot|^q) < \infty$  and the fact that $1+|x|\leq 2(1+|y|)$ for all $y\in B_\epsilon(x)$. In order to keep notation simple,  we write $w(x) := 1+ |x|$ for $x\in \real^d$. Putting all the estimates above together and using Theorem \ref{tail}(i), we have for any $\delta>0$ and $T \ge 2$
\begin{align} \label{b1}
	\Pp^\mu \left( \left|\frac 1 T \int_0^T \phi_{\epsilon}(x, X_t)\,\dup t \right| \ge \delta   \right)
	\leq 2 \exp\left(-C_2 T \epsilon^{d} \min\left\{w(x)^\frac{q}{2} \epsilon^{d} \delta^2,  \delta\right\}   \right)
\end{align}
for a suitable constant  $C_2 = C_2(\lambda, \| \rho \|_\infty) > 0$.  Since $\vol(Q)=h^d$ and $w(x_Q)^p\simeq 2^{np}$ as $x_Q\in B_n$, we get
\begin{align} \label{b2}
	\Phi_1 (T)
	\les h^d \sum_{Q \in \mathcal Q} w(x_Q)^p  \left|\frac 1 T \int_{0}^{T} \phi_\epsilon(x_Q, X_t)\,\dup t\right|.
\end{align}
Now we use the elementary inequality  $\Pp(\sum_{i} Z_i \ge \sum_i u_i) \le \sum_i \Pp(Z_i \ge u_i)$, and combine it with \eqref{b1} and \eqref{b2} to get
\begin{gather*}
\begin{split}
	\Pp^\mu \left( \Phi_1 (T) \ge \delta_{T}  \right)
	&\leq \sum_{Q \in \mathcal Q} \Pp^\mu \left( \left|\frac 1 T \int_0^T \phi_{\epsilon}(x_Q, X_t)\,\dup t \right|\ge \frac{\alpha_Q \delta_{T}}{h^d w(x_Q)^p}    \right)\\
	&\leq 2\sum_{Q \in \mathcal Q} \exp\left(-C_2 T \epsilon^{d} \min\left\{ \frac{\epsilon^{d} \alpha_Q^2 \delta_{T}^2}{h^{2d} w(x_Q)^{2p-\frac{q}{2}}},  \frac{\alpha_Q \delta_{T}}{h^d w(x_Q)^p}\right\}   \right).
\end{split}
\end{gather*}
In the above calculation, $\{\alpha_Q \}_{Q\in \mathcal Q}\subset (0,1)$ is any positive sequence with $\sum_{Q \in \mathcal Q} \alpha_Q = 1$, and
the parameter $\delta_T > 0$ will be chosen later in such a way that $\lim_{T\to\infty}\delta_T = 0$. We will use the following $\alpha_Q$'s. Let $\{ \beta_n \}_{0 \le n \le N}\subset (0,1)$ such that $\sum_{n = 0}^{N} \beta_n = 1$, and define
\begin{gather*}
	\alpha_Q := \frac{\beta_n}{\# \mathcal Q_n} = \frac{\beta_n}{\vol(B_n) /h^d} \quad \text{for } Q \in \mathcal Q_n.
\end{gather*}
Then,
\begin{gather} \label{basic}
	\Pp^\mu \left( \Phi_1 (T) \ge \delta_{T}  \right)
	\les h^{-d} T^{d \kappa} \cdot \max_{1 \le n \le N} \exp\left(-C_3  T \cdot \min\left\{ \frac{(\epsilon^{d} \delta_{T})^2 \beta_n^2 }{2^{(2p+ 2d - \frac{q}{2})n}},  \frac{\epsilon^{d} \delta_T \beta_n }{2^{(p+d)n}}\right\}   \right),
\end{gather}
where we used that $\vol(B_n) \simeq 2^{dn}$ and $w(x_Q) \simeq 2^n$ for $Q \in \mathcal Q_n$. Set $\zeta := \epsilon^d \delta_T \in (0, 1)$. By \eqref{basic}, we have
\begin{gather} \label{J1}
	\Pp^\mu \left( \Phi_1 (T) \ge \delta_{T}  \right)
	\les h^{-d} T^{d \kappa} \cdot  \exp\left(-C_3  T \cdot K(\zeta)   \right),
\end{gather}
where
\begin{gather*}
	K(\zeta)
	:= \max_{\beta_n \, : \sum_{n = 1}^{N} \beta_n = 1} \, \min_{1 \le n \le N}
	\min\left\{ \frac{\zeta^2 \beta_n^2 }{2^{(2p+ 2d - \frac{q}{2})n}},  \frac{\zeta \beta_n }{2^{(p+d)n}}\right\}.
\end{gather*}
We will now determine the order of magnitude of $K(\zeta)$. From its definition, we see that there exist positive numbers $\{ \beta_n \}_{0 \le n \le N}$ with $\sum_{n = 0}^{N} \beta_n = 1$ such that, for each $n = 0, \ldots, N$,
\begin{gather*}
	\frac{\zeta^2 \beta_n^2 }{2^{(2p+ 2d - \frac{q}{2})n}} \geq K(\zeta)
	\quad\text{and}\quad
	\frac{\zeta \beta_n }{2^{(p+d)n}} \ge K(\zeta).
\end{gather*}
Therefore,
\begin{gather*}
	\beta_n
	\geq \zeta^{-1} \max \left\{2^{(p+ d - \frac{q}{4})n} \sqrt{K(\zeta)},\: 2^{(p+d)n} K(\zeta)   \right\}.
\end{gather*}
Using the fact that $\sum_{n=1}^N \beta_n= 1$ yields
\begin{gather} \label{restriction}
	\sum_{n = 0}^{N} 2^{(p+d)n} \cdot \max \left\{2^{ - \frac{q}{4} n} \sqrt{K(\zeta)},\: K(\zeta)   \right\}
	\leq \zeta.
\end{gather}
It would be helpful to write the left-hand side of the previous inequality by using the  function $S : [0, \infty) \to \real$, which is defined as
\begin{gather*}
	S (r) := \sum_{n = 0}^{N} 2^{(p+d)n} \cdot \max \left\{ 2^{-\frac{q}{4}n}\sqrt{r},\: r\right\}.
\end{gather*}
Since  the function  $r\mapsto S(r)$ is strictly increasing, the equation $S(r) = \zeta$ has a unique solution, which we denote by $r_*$. In view of \eqref{restriction} we have $K(\zeta) \le r_*$. Furthermore, we see that $K(\zeta) = r_*$, if we choose $\beta_n^* := \zeta^{-1} 2^{(p+d)n} \cdot \max \left\{2^{\frac{q}{4} n} \sqrt{r_*},\,  r_*   \right\}$. Set
\begin{gather*}
	s := \min \left\{ \big\lfloor 2/q \log_2({r_*}^{-1}) \big\rfloor, N   \right\} \in \{0, \ldots, N   \}.
\end{gather*}
Then,
\begin{gather*}
	S (r_*)
	=  \sum_{n = 0}^{s} 2^{(p+d -\frac{q}{4})n} \sqrt{r_*} +  \sum_{n = s+1}^{N} 2^{(p+d)n} r_*
	= \zeta,
\end{gather*}
where we use the ``empty sum convention'', i.e.\ $\sum_{N+1}^N := 0$.  Consequently,
\begin{align}\label{max}
	\max \left\{ \sum_{n = 0}^{s} 2^{(p+d -\frac{q}{4})n} \sqrt{r_*}, \:  \sum_{n = s+1}^{N} 2^{(p+d)n} r_*  \right\}
	\simeq \zeta.
\end{align}
A   elementary  computation reveals that
\begin{gather} \label{ppp}
	K(\zeta)=r_*
	\simeq
	\begin{cases}
		\min\left\{ \zeta^2, \, 2^{-(p+d) N} \zeta\right\}
		& \text{if\ \ } p + d < \frac{q}{4},\\[6pt]
		\min\left\{ \frac{\zeta^2}{(\log ( \zeta^{-1}) \wedge N )^2}, \, 2^{-(p+d) N} \zeta    \right\}
		& \text{if\ \ } p + d = \frac{q}{4}, \\[6pt]
		\min\left\{ 2^{-(2p + 2d - \frac{q}{2}) N} \zeta^2, \, 2^{-(p+d) N} \zeta\right\}
		&\text{if\ \ } p + d \in (\frac{q}{4}, \frac{q}{2}], \\[6pt]
		2^{-(2p + 2d - \frac{q}{2}) N} \zeta^2
		&\text{if\ \ } p + d > \frac{q}{2}.
	\end{cases}
\end{gather}
Exemplarily, let us deal with the case $p+d > \frac{q}4$ in \eqref{ppp}; the other regimes follow similarly. If $s < N$, that is, if $r_* > 2^{-\frac{q}{2}N}$, then
\begin{gather*}
	\sum_{n = 0}^{s} 2^{(p+d -\frac{q}{4})n} \sqrt{r_*}
	< \sum_{n = 0}^{N} 2^{(p+d -\frac{q}{4})n} \sqrt{r_*}
	\les 2^{(p+d -\frac{q}{4})N} \sqrt{r_*}
	\les 2^{(p+d )N} r_*,
\end{gather*}
so that
\begin{gather*}
	\max \left\{ \sum_{n = 0}^{s} 2^{(p+d -\frac{q}{4})n} \sqrt{r_*}, \:  \sum_{n = s+1}^{N} 2^{(p+d)n} r_*  \right\}
	\simeq \sum_{n = s+1}^{N} 2^{(p+d)n} r_*
	\simeq 2^{(p+d )N} r_*.
\end{gather*}
If $s = N$, that is, if $r_* \le 2^{-\frac{q}{2}N}$, then
\begin{gather*}
	\max \left\{ \sum_{n = 0}^{s} 2^{(p+d -\frac{q}{4})n} \sqrt{r_*}, \:  \sum_{n = s+1}^{N} 2^{(p+d)n} r_*  \right\}
	= \sum_{n = 0}^{s} 2^{(p+d -\frac{q}{4})n} \sqrt{r_*}
	\simeq 2^{(p+d -\frac{q}{4})N} \sqrt{r_*}.
\end{gather*}
Combining these two cases with \eqref{max}, we immediately obtain \eqref{ppp}.

Recall that $N=N(T)$ was chosen in such a way that  $2^{N} \simeq T^\kappa$. In order to apply the Borel--Cantelli lemma, we need to find $\zeta(T)$ such that $K(\zeta) = C T^{-1} \log T$ for a sufficiently large $C$. Since $\epsilon(T) \simeq T^{-\theta}$ and $\zeta = \epsilon^d \delta_T$, it follows from \eqref{ppp} that  the choice
\begin{gather} \label{pppp}
	\delta_T
	:=
	\begin{cases}
	T^{-[1 - (p+d)\kappa - d\theta]} \log T
	& \begin{aligned}
		&\text{if\ \ } p+d \le \tfrac{q}{2},\;  p + d \ne \tfrac{q}{4},\; \min\{p+d, \tfrac{q}{4} \} \kappa \ge \tfrac12,\\
		&\,\,\,\,\,\,\text{or\ \ } p+d = \tfrac{q}{4}, (p+d)\kappa> \tfrac12;
	\end{aligned}\\[10pt]
		T^{-\frac12 + d\theta} (\log T)^{\frac32}
	& \,\text{if\ \ } p + d = \frac{q}{4},\; (p+d)\kappa \le \frac12;\\[10pt]
		T^{-[\frac12 -(p+d - \frac{q}4)_+ \kappa - d\theta]} (\log T)^{\frac12}
	&\begin{aligned}
		&\text{if\ \ } p+d \le \tfrac{q}{2}, \;  p + d \ne \tfrac{q}{4},\; \min\{p+d, \tfrac{q}{4} \} \kappa < \tfrac12,\\
		&\,\,\,\,\,\,\text{or\ \ } p+d > \tfrac{q}{2}
	\end{aligned}
\end{cases}
\end{gather}
achieves $K(\epsilon^d \delta_T) \simeq T^{-1} \log T$. Of course, the parameters $\kappa, \theta >0$ have to be sufficiently small, so that all powers of $T$ appearing in the expression for $\delta_T$ are negative.

Once $\kappa$ and $\theta$ are fixed,  we can choose $h = h(T) = \entier{T}^{-\upsilon}$ in \eqref{J2}  with a sufficiently large integer $\upsilon>0$, which ensures that
\begin{gather} \label{J2''}
	\Phi_2(T)
	= o(\delta_T) \quad\text{as\ }  T \to \infty.
\end{gather}
As $h=h(T)$ is determined, we can combine \eqref{J1} and \eqref{ppp} with \eqref{pppp}, and find a constant $C_4>0$ such that
\begin{gather} \label{J1'}
	\Pp^\mu \left( \Phi_1(T) \ge C_4 \delta_T \right)
	\les T^{-2},
\end{gather}
where $\delta_T$ is defined by \eqref{pppp}.

Set $T := T_k := k$ for $ k\in \nat$ with $k\ge 2$. From \eqref{I1'}, \eqref{J2''}, \eqref{J1'} and the Borel-Cantelli lemma, we obtain
\begin{align} \label{as1}
	\limsup_{k \to \infty} \frac{\Phi(T_k)}{\delta_{T_k}}
	\leq M \quad\text{a.s.}
\end{align}
for some constant $M>0$. If $T \ge 2$ is not an integer, then, by construction, $N(T) = N(\entier{T})$ and $\epsilon(T) = \epsilon(\entier{T})$. Because of the definition of $\Phi$,
\begin{gather}
\begin{split}\label{as2}
	\Phi(T)
	\leq &\frac{\entier{T}}{T} \sum_{n=0}^{N} 2^{pn}\int_{B_n}\left|\frac 1 {\entier{T}} \int_0^{\entier{T}}\phi_\epsilon(x,X_t)\,\dup t\right| \dup x + \frac{1}{T} \sum_{n=0}^{N} 2^{pn}\int_{B_n}\left| \int_{\entier{T}}^{T} \phi_\epsilon(x,X_t)\,\dup t\right|\dup x\\
	\lesssim &\frac{\entier{T}}{T} \Phi(\entier{T}) + \frac{1}{T} \cdot \sum_{n = 0}^{N} 2^{(p+d) n} \| \phi_\epsilon \|_\infty\\ \leq &\Phi(\entier{T}) + O(T^{-[1  - (p + d)\kappa -  d\theta]}).
\end{split}
\end{gather}
Combining \eqref{I1}, \eqref{as1} and \eqref{as2}, we deduce that there is some constant $M'>0$ such that
\begin{align} \label{A}
\limsup_{T \to \infty} \frac{\Jcal_1(T)}{\delta_T + T^{-[1 - (p + d)\kappa -  d\theta ]}} \le M' \, \, \text{ a.s.}
\end{align}

\medskip\noindent
\textbf{Estimating  $\Jcal_2(T)$.}
Now we bound the term $\Jcal_2(T)$ from \eqref{J2T}. It follows from \eqref{pmu} that, for any Borel set $A \subseteq \real^d$ and $\epsilon \in (0, 1)$,
\begin{align*}
\big| \left( \mu_T*\Lscr_{\epsilon \xi} \right) (A)- \left( \mu*\Lscr_{\epsilon \xi} \right)(A) \big| \le \frac1T \int_{0}^{T} \Pp (X_t + \epsilon \xi \in A) \, \dup t + \Pp(X+\epsilon \xi \in A),
\end{align*}
where $X$ denotes a random variable that is independent of $\xi$ and has the law $\mu$. Then, for each $n, \ell \in \nat$ and $\epsilon \in (0, 1)$,
\begin{gather}
\begin{split} \label{sum}
& \sum_{F \in \mathcal P_\ell} \big| \left( \mu_T*\Lscr_{\epsilon \xi} \right) (2^n F\cap B_n)- \left( \mu*\Lscr_{\epsilon \xi} \right)(2^n F\cap B_n) \big|\\
& \le \sum_{F \in \mathcal P_\ell} \frac1T \int_{0}^{T} \Pp (X_t + \epsilon \xi \in 2^n F\cap B_n) \, \dup t + \sum_{F \in \mathcal P_\ell} \Pp(X+\epsilon \xi \in 2^n F\cap B_n)\\
& = \frac1T \int_{0}^{T} \Pp (X_t + \epsilon \xi \in B_n) \, \dup t + \Pp(X+\epsilon \xi \in B_n) \le \frac1T \int_{0}^{T} \I_{\{X_t \in B_n^\epsilon  \}} \, \dup t + \mu(B_n^\epsilon),
\end{split}
\end{gather}
where we used that if $X_t \notin B_n^\epsilon$ then $\Pp (X_t + \epsilon \xi \in B_n) = 0$, and so $\Pp\left(X_t + \epsilon \xi \in B_n\right) \leq \I_{\left\{X_t\in B_n^\epsilon\right\}}$.

Recall that $\epsilon(T) \simeq T^{-\theta}$ and $N(T) \simeq \log_2 T$.  If $T\gg 1$ is sufficiently large, then there exists a constant $c>0$ such that for all $n > N(T)$ and $x \in B_n^\epsilon$,
\begin{gather*}
	c\, 2^n \leq |x| \leq c^{-1} 2^n.
\end{gather*}
Combining this with \eqref{sum}, we obtain for sufficiently large $T$ there exists $c'>0$ such that
\begin{gather} \label{J2'}
\begin{split}
\Jcal_2(T) &\le (1 - 2^{-p})^{-1} \sum_{n > N} 2^{pn} \left( \frac1T \int_{0}^{T} \I_{\{X_t \in B_n^\epsilon  \}} \, \dup t + \mu(B_n^\epsilon) \right)\\
& \lesssim \sum_{n > N}\left\{  \frac1T \int_{0}^{T} |X_t|^p \, \I_{\{c\, 2^n \leq |X_t| \leq c^{-1} 2^n  \}}\, \dup t   + \mu \left(|\cdot|^p \I_{\{c\, 2^n \leq |\cdot| \leq c^{-1} 2^n  \}}\right) \right\}\\
& \les \frac1T \int_{0}^{T} |X_t|^p \, \I_{\{|X_t| \ge c' T^\kappa  \}}\, \dup t  + \mu \left(|\cdot|^p \I_{\{|\cdot| \ge c' T^\kappa  \}}\right) =: \Psi_1(T) + \Psi_2(T).
\end{split}
\end{gather}
In the last line we use the elementary estimate $\| \sum_{n\ge0} \I_{\left\{c\, 2^n \leq |\cdot| \leq c^{-1} 2^n\right\}} \|_\infty < \infty$.

The term $\Psi_2(T)$ is deterministic.  Since $\mu(|\cdot|^q) < \infty$, the Markov inequality shows that
\begin{align} \label{psi1}
\Psi_2 (T) \les T^{-(q - p)\kappa}.
\end{align}
Let us turn to $\Psi_1(T)$. Split this term into two parts
\begin{gather*}
\Psi_1 (T) = \frac1T \int_{0}^{2} |X_t|^p \, \I_{\{|X_t| \ge c' T^\kappa  \}}\, \dup t + \frac1T \int_{2}^{T} |X_t|^p \, \I_{\{|X_t| \ge c' T^\kappa  \}}\, \dup t =: \Psi_{1, 1} (T) + \Psi_{1, 2}(T).
\end{gather*}
Clearly,
\begin{gather*}
\Psi_{1, 1} (T) \les T^{-1 - (q - p)\kappa}  \int_{0}^{2} |X_t|^q \, \dup t.
\end{gather*}
Using the Tonelli-Fubini theorem,
\begin{align*}
\Ee^\mu \left[ \int_{0}^{2} |X_t|^q \, \dup t \right] = \int_{0}^{2} \Ee^\mu \left[  |X_t|^q\right] \,  \dup t < \infty,
\end{align*}
which implies that $\displaystyle\int_0^2 |X_t|^q\,\dup t<\infty$ almost surely. Thus,
\begin{align} \label{psi11}
\limsup_{T \to \infty} \frac{\Psi_{1, 1} (T)}{T^{-1 - (q - p)\kappa}}  < \infty \, \, \text{ a.s.}
\end{align}
Using the Tonelli-Fubini theorem once again, we get for any $\eta >1$
\begin{align*}
\Ee^\mu \left[ \int_{2}^{\infty} \frac{|X_t|^q}{t (\log t)^\eta} \, \dup t \right] = \int_{2}^{\infty} \frac{\Ee^\mu \left[  |X_t|^q\right]}{t (\log t)^\eta} \,  \dup t < \infty.
\end{align*}
Consequently, $\displaystyle\int_{2}^{\infty} \frac{|X_t|^q}{t (\log t)^\eta} \, \dup t < \infty$ almost surely. Since
\begin{gather*}
\Psi_{1, 2} (T) \les T^{- (q - p)\kappa} \cdot \frac1T \int_2^{T} |X_t|^q \, \dup t \les T^{- (q - p)\kappa} (\log T)^\eta \cdot  \int_2^{T} \frac{|X_t|^q}{t(\log t)^\eta} \, \dup t,
\end{gather*}
we deduce that for any $\eta > 1$,
\begin{align} \label{psi12}
\limsup_{T \to \infty} \frac{\Psi_{1, 2} (T)}{T^{- (q - p)\kappa} (\log T)^\eta}  < \infty \, \, \text{ a.s.}
\end{align}
 From  \eqref{J2'}--\eqref{psi12}, we obtain that for any $\eta > 1$,
\begin{gather} \label{B}
\limsup_{T \to \infty} \frac{\Jcal_2 (T)}{T^{- (q - p)\kappa} (\log T)^\eta}  < \infty \, \, \text{ a.s.}
\end{gather}

Combining \eqref{C},  \eqref{J1T}, \eqref{J2T} and  \eqref{A} with \eqref{B}, and optimizing in $\kappa, \theta>0$, we conclude that
\begin{gather*}
\limsup_{T \to \infty} \frac{\Tcal_p(\mu_T,\mu)}{R_\eta(T)}  < \infty \, \, \text{ a.s.},
\end{gather*}
where the rate function $R_\eta(T)$ is given by \eqref{rate1}. To see this, we deal with the case $p+d = \frac{q}{4}$; the remaining regimes follow with similar arguments. Indeed, it suffices to choose $\kappa, \theta>0$ so as to maximize
\begin{gather*}
	\min \left\{ 1 - \max \left\{ (p+d)\kappa,\, \frac12 \right\}- d \theta ,\:   (q - p) \kappa,\:  p \theta  \right\}.
\end{gather*}
 The optimal choice can be taken as $\kappa = \theta = \frac{1}{2(p+d)}$. This gives precisely the rate function \eqref{rate1}.

If assumption \textbf{(H2)} is replaced by \textbf{(H3)}, then the Bernstein inequality improves from the form of Theorem \ref{tail}\eqref{Bern-IP} to that in Theorem \ref{tail}\eqref{Bernstein}. Consequently, in \eqref{b1} $q/2$ should be replaced by $q$, causing the same change in \eqref{pppp}. The same argument, with \eqref{B} unchanged, now yields
\begin{gather*}
	\limsup_{T \to \infty} \frac{\Tcal_p (\mu_T,\mu)}{\widetilde R_\eta(T)}  < \infty \quad \text{a.s.}
\end{gather*}
where the rate function $\widetilde R_\eta(T)$ is given by \eqref{rate2}.

\section{Examples}
In this section, we present several examples to illustrate our main results and subsequently compare them with related results in the literature.

\subsection{Diffusions}

Consider the following SDE
\begin{gather}\label{SDE-L}
	\dup X_t = b(X_t) \, \dup t + \sigma(X_t) \, \dup W_t
\end{gather}
with measurable coefficients $b:\real^d\to\real^d$, $\sigma:\real^d \to \real^{d \times d}$ and a standard $d$-dimensional Brownian motion $(W_t)_{t\ge0}$. The recent paper \cite{CP} establishes convergence rates in the $\Wa_2$-distance assuming uniform dissipativity and Lipschitz continuity of the coefficients. The argument in \cite{CP} relies essentially on assumption \textbf{(H1)}, so Theorem~\ref{D} applies directly to that setting and yields, in addition, bounds in the $\Tcal_p$-distance for general $p>0$. In the following example, we consider a slightly different case where \textbf{(H1)} holds under a weaker long-distance dissipativity condition,  cf.\ (II).

\begin{example}\label{diff}
	 Assume that $\sigma(x)\in\real^{d\times d}$ is uniformly elliptic, i.e.,\ there is a constant $\theta > 0$ such that for all $x\in\real^d$,  $\sigma(x)\sigma^\top(x) \ge \theta \, \mathds{I}_{d\times d}$.  Write
	\begin{gather*}(\sigma \sigma^\top )(x) = \theta \, \mathds{I}_{d \times d} + (\hat \sigma \hat \sigma^\top)(x),\quad x\in\real^d,\end{gather*}
	where $\hat \sigma:\real^d \to \real^{d \times d}$ is measurable. Assume also that there are constants $L,\alpha_1,\alpha_2,R>0$ such that for any $x,y\in\real^d$,
	\begin{enumerate}
		\item[$({\rm I})$] $\|\hat \sigma(x)-\hat \sigma(y)\|_{\rm{HS}}\le L|x-y|,$
		\item[$({\rm II})$] $2\langle b(x)-b(y),x-y\rangle + \|\hat \sigma(x)-\hat \sigma(y)\|_{\rm{HS}}^2\le \phi(|x-y|)\, |x-y|,$
	\end{enumerate}
	where $\| \cdot \|_{\rm{HS}}$ denotes the Hilbert-Schmidt norm, and the function $\phi : [0, \infty) \to \real$ is defined by
	\begin{gather*}
		\phi(r)
		:= r \cdot \left[ \alpha_1  \, \I_{[0,R]}(r) + \left\{ \alpha_1 - (\alpha_1+\alpha_2) \left( \frac{r}{R} - 1 \right) \right\}
		 \I_{(R,2R]}(r) - \alpha_2  \, \I_{(2R,\infty)}(r) \right].
	\end{gather*}
Then, the SDE \eqref{SDE-L} admits a unique strong solution $(X_t)_{t\ge0}$, and has a unique  invariant probability measure $\mu$ satisfying $\mu(|\cdot|^q)<\infty$ for some $q > 2$. Moreover, we have for $\Ee^\mu[\Tcal_p(\mu_T,\mu)]$ and any $p<q$ the bound \eqref{D-W}.
\end{example}

\begin{proof}
	It is shown in \cite[Theorem 3.1 and Corollary 3.4]{HLM} that, under assumptions on the coefficients $b$ and $\sigma$ in the example, the SDE \eqref{SDE-L} has a unique strong  solution $(X_t)_{t\ge0}$, and has a unique  invariant probability measure $\mu$ such that \textbf{(H1)} is satisfied. In order to apply Theorem \ref{D}, we   need  to show only that  $\mu(|\cdot|^q)<\infty$ for some $q > 2$.
	
	For $q> 2$, set $g(x)=(1+|x|^2)^{\frac q 2}$. Then
	\begin{gather*} \nabla g(x)=  q (1+|x|^2)^{-1} g(x) \, x, \quad \nabla^2 g(x)=  q (1+|x|^2)^{-1} g(x) \, \left( \mathds{I}_{d\times d} + \frac{q-2 }{1+|x|^2} \, x x^\top \right),\end{gather*}
	Since  the generator of the solution of the SDE  \eqref{SDE-L} is
	\begin{gather*}
	\Lcal f (x) = \langle b(x), \nabla f (x)   \rangle  + \frac12 \text{Tr} \big( (\sigma \sigma^\top)(x) \nabla^2 f(x)  \big), \quad f \in C^2(\real^d),
	\end{gather*}
	a short computation yields
	\begin{equation*} \label{5.2}
		\begin{split}
			\Lcal g(x)&= q (1+|x|^2)^{-1} g(x) \cdot \left[ \langle b(x),x\rangle +\frac 1 2  \text{Tr}\big( (\sigma \sigma^\top)(x) \big)+\frac{q-2}{2(1+|x|^2)}  |\sigma^\top(x) x |^2 \right]\\
			&\le q (1+|x|^2)^{-1} g(x) \cdot \left[ \langle b(x),x\rangle +\frac {q - 1} 2  \|\hat \sigma(x)\|_{\text{HS}}^2 + \frac{\theta(q + d - 2)}{2} \right],
	\end{split}\end{equation*}
where  the last inequality is due to
\begin{gather*}
	\text{Tr}((\sigma \sigma^\top)(x))
	=  \theta d +\|\hat \sigma(x)\|_{\text{HS}}^2,
	\quad
	|\sigma^\top(x)x|^2
	\le \theta |x|^2 + \|\hat \sigma(x)\|_{\text{HS}}^2|x|^2.
\end{gather*}
	By the assumptions (I), (II) and the triangle inequality, for any $\epsilon \in (0, 1)$ and $x \in \real^d$,
	\begin{equation*}
		\begin{split}
			& \langle b(x),x\rangle +\frac {1} 2  \|\hat \sigma(x)\|_{\text{HS}}^2\\
 &\le \langle b(x), x \rangle  + \frac {1} 2 \big[ (1 + \epsilon) \|\hat \sigma(x)- \hat \sigma(0)\|_{\text{HS}}^2 + (1 + \epsilon^{-1}) \|\hat \sigma(0)\|_{\text{HS}}^2\big]\\
			& \le  \langle b(x)-b(0),x\rangle  +  \frac {1} 2 \|\hat \sigma(x)-\hat \sigma(0)\|_{\text{HS}}^2   +  \langle b(0),x \rangle  + \frac{\epsilon L^2}{2} |x|^2  + \frac{ (1 + \epsilon^{-1})}{2} \|\hat \sigma(0)\|_{\text{HS}}^2\\
			& \le \frac{\phi(|x|)|x|}{2} + \frac{\epsilon (1 +   L^2)}{2} |x|^2 +  \frac{\epsilon^{-1} |b(0)|^2}{2} + \frac{ 1 + \epsilon^{-1}}{2} \|\hat \sigma(0)\|_{\text{HS}}^2  ,
		\end{split}
	\end{equation*}
	and
	\begin{equation*} \label{5.4}
		\begin{split}
			\|\hat \sigma(x)\|_{\text{HS}}^2 &\le  (L |x| + \|\hat \sigma(0)\|_{\text{HS}} )^2 \le (1 + \epsilon)L^2 |x|^2 + (1 + \epsilon^{-1}) \|\hat \sigma(0)\|_{\text{HS}}^2.
		\end{split}
	\end{equation*}
 The above estimates    and  the definition of $\phi$ show that  for any $2 < q < 2 + \alpha_2 L^{-2} $,
	one can choose $\epsilon$ sufficiently small so that there exist constants $K_1, K_2 > 0$
so that for all $x\in \real^d$,
	\begin{equation*}
		\Lcal g(x) \le K_1 - K_2 \, g(x).
	\end{equation*}
Together  with the invariance of $\mu$, this implies that $\mu(|\cdot|^q)<\infty$.
\end{proof}

For the classical Ornstein--Uhlenbeck process, our approach yields the $p$-Wasserstein convergence rate for the empirical measure. The next example treats a more general case, which has originally been considered by Wang in \cite[Example 1.4]{W22}.

\begin{example}
	Suppose that
	\begin{gather}\label{SDE-OU}\sigma(x)=\mathds{I}_{d \times d}, \quad b(x)=-\kappa \alpha |x|^{\alpha-2}x+\nabla \varphi (x),\end{gather}
	where $\alpha>1$, $\kappa>0$, and $\varphi \in C^1(\real^d)$ satisfies $\|\nabla \varphi \|_\infty<\infty$. Then the SDE \eqref{SDE-L} has a unique strong solution $(X_t)_{t\ge0}$, and has a unique  invariant probability measure
	\begin{gather*}
		\mu(\dup x)
		:= \frac{1}{Z} \eup^{-\kappa |x|^\alpha+ \varphi (x)} \, \dup x,
		\qquad
		Z : = \int_{\real^d} \eup^{-\kappa |y|^\alpha+ \varphi (y)} \, \dup y.
	\end{gather*}
	Define, for
$R>1$, a smooth function $J(x)$ satisfying $J(x) = \eup^{|x|}$ for $|x|\ge R$ and $J(x)\ge 1$ for $|x|<R$. By $\|\nabla \varphi\|_\infty<\infty$ and $\alpha>1$, we can show that, for
$R>1$ large enough,
	\begin{gather*}
		\Lcal J \le -c_1 J + c_2 \I_{B_R(0)}
	\end{gather*}
	holds with some constants $c_1, c_2>0$,  where $\Lcal := \frac12 \Delta+ \langle b, \nabla\rangle $ is the generator of the process $(X_t)_{t\ge0}$ that is symmetric with respect to $\mu$. From Lyapunov's criterion (see, e.g., \cite[Theorem~4.6.2]{BGL14}), we can conclude that the assumption~\textbf{(H3)} is satisfied. Consequently, Theorem \ref{S} and Theorem \ref{thm3} apply, yielding bounds for $\Ee^\mu [\Tcal_p(\mu_T, \mu)]$ as well as almost sure bounds for $\Tcal_p(\mu_T, \mu)$.
\end{example}

When $p = 2$, by a different method, \cite[(1.20)]{W22} establishes that
\begin{gather}\label{Wang-OU}
\mathds{E}^{\mu}[\Tcal_2(\mu_T, \mu)] \lesssim
\begin{cases}
	T^{-\frac{2(\alpha - 1)}{(d - 2)\alpha + 2}} & \text {if \, $4(\alpha - 1) < d\alpha$,} \\
	T^{-1} \log(1 + T) & \text {if \, $4(\alpha - 1) = d\alpha$,} \\
	T^{-1} & \text {if \, $4(\alpha - 1) > d\alpha$.}
\end{cases}
\end{gather}
Our results here extend the findings of \cite{W22} to any $p>0$ and yield sharper bounds for $p=2$, provided that $1<\alpha<\alpha_c(d)$, where
\begin{gather*}
\alpha_c(d) :=
\begin{cases}
	\frac{6}{6-d} & \text {if \, $d<4$,}\\[4pt]
	\frac{d+2}{2} & \text {if \, $d \ge 4$.}
\end{cases}
\end{gather*}
	 Furthermore, let $\kappa=1$, $\alpha=2$ and $\varphi=0$ in \eqref{SDE-OU}. Then we get the classical Ornstein--Uhlenbeck process. Combining \eqref{S1} with \eqref{Wang-OU}, we obtain the best known bound for its empirical measure under the $\Wa_2$ distance.
\begin{corollary}
	Let $(X_t)_{t\ge0}$ be the $d$-dimensional Ornstein--Uhlenbeck process defined by
	\begin{gather*}
	\dup X_t = -X_t \, \dup t+  \, \dup W_t.
	\end{gather*}
	Then its empirical measures $\mu_T$ satisfy
	\begin{gather*}
	\Ee^\mu\left[ \Tcal_2(\mu_T,\mu)\right]
	\lesssim
	\begin{cases}
		T^{-1}			& \text{if\ \ } d = 1,\\
		T^{-1} \log T	& \text{if\ \ } d = 2,\\
		T^{-\frac 1 2} & \text{if\ \ } d = 3,\\
		T^{-\frac 1 2}\log T & \text{if\ \ }d = 4,\\
		T^{-\frac 2 d}		& \text{if\ \ }d \geq  5.
	\end{cases}
	\end{gather*}
\end{corollary}

\subsection{Underdamped Langevin Dynamics}\label{Lan}

Let us consider the following underdamped Langevin dynamics $(X_t)_{t\ge0} = (Y_t, Z_t)_{t\ge0}$
on $\real^n\times \real^n$:
\begin{equation} \label{Langevin}
	\begin{cases}
		\dup Y_t = Z_t \, \dup t,\\
		\dup Z_t =- \big( Z_t + \nabla V(Y_t)  \big) \, \dup t + \sqrt{2 } \, \dup W_t,\\
	\end{cases}
\end{equation}
where $V \in C^2(\real^n)$ is  a confining potential,    and $(W_t)_{t\ge0}$ is the standard $n$-dimensional Brownian motion. Under mild assumptions (see \cite[Chapter~6]{Pa}),   the process $(X_t)_{t\ge0}$ defined by \eqref{Langevin} has a unique invariant probability measure given by
\begin{gather*}
\mu(\dup y, \dup z) : = \mu_V(\dup y) \, \Normal (\dup z) \in \mathscr P(\real^{2n}),
\end{gather*}
where $\Normal$ is the standard Gaussian measure on $\real^n$ and
\begin{gather*}
\mu_V (\dup y) := \frac{1}{Z_V} \eup^{-V(y)} \, \dup y, \qquad  Z_V : = \int_{\real^n} \eup^{-V(x)} \, \dup x.
\end{gather*}
 Let $(P_t)_{t\ge0}$ be the associated Markov semigroup of the process $(X_t)_{t\ge0}$, whose infinitesimal generator is
\begin{gather*}
\Lcal  =  \Delta_z - \langle  z + \nabla V(y), \nabla_z \rangle + \langle z, \nabla_y \rangle, \quad (y, z) \in \real^n \times \real^n.
\end{gather*}
In order to apply our main results to \eqref{Langevin}, we have to verify the assumptions \textbf{(H1)}-\textbf{(H3)}.

We first note that, in general, assumption \textbf{(H3)} is not expected to hold for this dynamics. Indeed, for any potential $V$ satisfying $\mu_V(|y|^2) < \infty$, denote $m_V: = \int_{\real^n} y\mu_V(\dup y) \in \real^n$. Let us consider the function $f(y, z) = \langle a, y - m_V\rangle$, where $a \in \real^n$ is a fixed nonzero vector. Clearly,  $f$ is in the (extended) domain of $\Lcal$ and
with $(\Lcal f)(y, z) = \langle a, z\rangle$. Moreover, it's obvious that $\mu(f) = \langle a, \mu_V(y - m_V)\rangle = 0$. Note that
\begin{gather*}
	\langle  \Lcal f, f \rangle _{L^2(\mu)}
	= \int_{\real^n} \langle a,y-m_V\rangle\,\mu_V(\dup y) \cdot \int_{\real^n} \langle a,z\rangle\, \Normal(\dup z)
	= 0;
\end{gather*}
thus \textbf{(H3)} does not hold.

The assumptions \textbf{(H1)} and \textbf{(H2)} are known to hold under suitable conditions on $V$, see \cite[Section 2]{Sc} and \cite{CLW}, respectively; a very recent work \cite{M} indicates validity of \textbf{(H2)} in the non-equilibrium case. Therefore, our results apply to the underdamped Langevin dynamics \eqref{Langevin}, and yield quantitative Wasserstein convergence rates for the empirical measures. From now on, we focus on the case where \textbf{(H2)} holds. In this setting, we may use the main theorem of \cite{CLW}. As in \cite{CLW}, we impose the following assumptions on $V\in C^2(\real^n)$:
\begin{enumerate}
	\item[(i)]
	Assume that the potential $V$ satisfies a Poincar\'e inequality: there exists a constant $c_1 > 0$ such that for any $f \in H^1( \mu_V)$ (this is the standard Sobolev space) with $\mu_V(f) = 0$,
	\begin{gather*}
		\mu_V(f^2) \leq c_1 \mu_V \left(|\nabla f|^2\right).
	\end{gather*}
	
	\item[(ii)]
	The potential satisfies $V \in C^2(\real^n)$, and there exist constants $c_2 \ge 1$ and $\delta \in (0, 1)$ such that for any $y \in \real^n$,
	\begin{align*}
		\sum_{i, j = 1}^{n} |\partial_{ij} V(y)|^2
		\le  c_2(1 + |\nabla V(y)|^2),\quad
		\Delta V(y) \leq  c_2   + \frac{\delta}{2} \, |\nabla V(y)|^2.
	\end{align*}
	
	\item[(iii)]
	The embedding $H^1(\mu_V) \hookrightarrow L^2(\mu_V)$ is compact.
\end{enumerate}
The assumption (i) implies the exponential integrability of Lipschitz functions with respect to $\mu_V$, see e.g.\ \cite[Section 4.4.2]{BGL14}, and therefore the measure $\mu = \mu_V \otimes \Normal$ has finite moments of any order. The assumption (iii) is satisfied, if
\begin{gather*}
	\lim_{|y| \to \infty} \frac{V(y)}{|y|^\beta} = \infty
\end{gather*}
for some $\beta > 1$, see \cite{CLW} for details. Under the above assumptions, Theorem~1 in \cite{CLW} yields \textbf{(H2$\mbox{}^\prime$)}, which implies \textbf{(H2)}, cf.\ the discussion in Section \ref{subsec-more}. Thus, we obtain the following corollary.

\begin{corollary} \label{lang}
	Assume that the potential $V$ satisfies \textup{(i)--(iii)}. Then, for the underdamped Langevin dynamics $(X_t)_{t\ge0}$ given by \eqref{Langevin} and any $p>0$, we have
	\begin{gather*}
		\Ee^\mu\left[\Tcal_p (\mu_T,\mu)\right]
		\lesssim
		\begin{cases}
			T^{-\frac 1 2}		& \text{if\ \ } p > \frac32 n,\\
			T^{-\frac 1 2}\log T	& \text{if\ \ } p = \frac32 n,\\
			T^{-\frac {p} {3n}}	& \text{if\ \ } p < \frac32 n,
		\end{cases}
	\end{gather*}
	and for any $\eta >1$,
	\begin{equation*}
		\limsup_{T \to \infty} \frac{\Tcal_p (\mu_T,\mu)}{T^{-\frac{p}{2(p + 2n)}} (\log T)^\eta}  < \infty \, \, \text{ a.s.}
	\end{equation*}
\end{corollary}

To further illustrate our results, we consider \cite[Example 3.1]{W25} in the special case where $m = n$. The potential is given by
\begin{gather*}
	V(y) = \psi(y) + \left(1 + c |y|^2\right)^\theta, \quad y \in \real^n,
\end{gather*}
where $\psi \in C_b^2(\real^n)$, and constants $c > 0$ and $\theta> \frac12$.  Then, assumptions (i)-(iii) are satisfied (for the verification of (i), see \cite[Example 2.1]{W25}).  Therefore, for $p =2$, Corollary \ref{lang} yields
\begin{gather*}
	\Ee^\mu\left[\Tcal_2 (\mu_T,\mu)\right]
	\lesssim
	\begin{cases}
		T^{-\frac 1 2}		& \text{if\ \ } n = 1,\\
		T^{-\frac {2} {3n}}& \text{if\ \ } n \ge 2,
	\end{cases}
\end{gather*}
which improves the bound in \cite[(3.24)]{W25} in the case $m = n\ge 2$ and $\alpha =1$.

\begin{ack}

	R.L.\ Schilling is supported by the ScaDS.AI centre (TU Dresden, Universit\"at Leipzig).
	J.\ Wang is supported by the National Key R\&D Program of China (Grant No. 2022YFA1006003) and the NNSFC (Grant Nos. 12225104 and 12531007).
	B.\ Wu is supported by the National Key R \& D Program of China (Grant Nos.~2023YFA1010400 and 2022YFA1006003), the NNSFC (Grant No.~12401174), the  Natural Science Foundation-Fujian (Grant No.~2024J08051), Fujian Alliance of Mathematics (Grant No.~2023SXLMQN02), the Education and Scientific Research Project for Young and Middle-aged Teachers in Fujian Province of China (Grant No.~JAT231014) and the Alexander von Humboldt Foundation.
J.-X. Zhu acknowledges support from  the Natural Science Foundation of Shanghai (Grant No. 25ZR1402414) and the NNSFC (Grant Nos. 12271102 and 12501185).
\end{ack}

\vskip 0.3truein
{\bf Ren\'e L. Schilling:}\,\,
Institut f\"{u}r Mathematische Stochastik, Fakult\"{a}t Mathematik, TU Dresden,  Dresden 01062, Germany.
\newline
\quad\texttt{rene.schilling@tu-dresden.de}

\bigskip
{\bf Jian Wang:}\,\, School of Mathematics and Statistics \& Key Laboratory of Analytical Mathematics and Applications
 (Ministry of Education) \& Fujian Provincial Key Laboratory of Statistics and Artificial Intelligence, Fujian Normal University, Fuzhou 350007, China. \newline
\quad \texttt{jianwang@fjnu.edu.cn}

\bigskip
{\bf Bingyao Wu:}\,\,  School of Mathematics and Statistics \& Key Laboratory of Analytical Mathematics and Applications
 (Ministry of Education), Fujian Normal University, Fuzhou 350007, China; \ \

\emph{Current address}: Institut f\"{u}r Mathematische Stochastik, Fakult\"{a}t Mathematik, TU Dresden,  Dresden 01062, Germany. \newline
\quad\texttt{bingyaowu@163.com}

\bigskip
{\bf Jie-Xiang Zhu:}\,\, Department of Mathematics, Shanghai Normal University, Shanghai 200234, China. \newline
\quad\texttt{jiexiangzhu7@gmail.com}

\end{document}